\documentclass[ejs,noinfoline]{imsart} 

\setpkgattr{copyright}{text}{}
\setpkgattr{issuedata}{text}{}

\usepackage{amsthm, amsmath, amssymb,amsfonts,bm,bbm,xcolor,mathtools,comment,colordvi} 

\usepackage[normalem]{ulem}

\RequirePackage[numbers,sort]{natbib}
\RequirePackage[colorlinks,citecolor=blue,urlcolor=blue]{hyperref}
\usepackage{graphicx}

\pubyear{2020}
\volume{0}
\issue{0}
\firstpage{1}
\lastpage{8}

\startlocaldefs

\RequirePackage[OT1]{fontenc}

\usepackage{amsfonts, amssymb}
\usepackage{graphicx, color, latexsym}
\usepackage{tikz}
\usepackage[bf,footnotesize]{caption}

\usepackage{esint}
\usepackage{dsfont}

\newtheorem{lem}{Lemma}

\theoremstyle{remark}
\newtheorem{remark}{Remark}

{}
\def\1{1\!{\rm l}}

\newcommand{\leqa}{\lesssim}
\newcommand{\geqa}{\gtrsim}

\newcommand{\EM}{\ensuremath}

\newcommand{\al}{\alpha}

\newcommand{\ga}{\gamma}

\newcommand{\la}{\lambda}

\newcommand{\te}{\theta}
\newcommand{\ta}{\vartheta}
\newcommand{\veps}{\varepsilon}
\newcommand{\vphi}{\varphi}

\newcommand{\cA}{\EM{\mathcal{A}}}

\newcommand{\cC}{\EM{\mathcal{C}}}

\newcommand{\cF}{\EM{\mathcal{F}}}

\newcommand{\cI}{\EM{\mathcal{I}}}

\newcommand{\cM}{\EM{\mathcal{M}}}
\newcommand{\cN}{\EM{\mathcal{N}}}
\newcommand{\cP}{\EM{\mathcal{P}}}

\newcommand{\cS}{\EM{\mathcal{S}}}
\newcommand{\cT}{\EM{\mathcal{T}}}
\newcommand{\cU}{\EM{\mathcal{U}}}
\newcommand{\cV}{\EM{\mathcal{V}}}

\newcommand{\psg}{{\langle}}
\newcommand{\psd}{{\rangle}}

\DeclareMathAlphabet{\mathpzc}{OT1}{pzc}{m}{it}

\newcommand{\noi}{\noindent}

\newcommand{\RR}{\mathbb{R}}

\newcommand{\given}{\,|\,}

\definecolor{blendedblue}{rgb}{0.2,0.2,0.7}

\newcommand{\rn}{\sqrt{n}}

\newcommand{\vp}{\vspace{.5cm}}

\newcommand{\mockalph}[1]{}

\theoremstyle{plain}
\newtheorem{thm}{Theorem}

\endlocaldefs

\begin{document}

\begin{frontmatter}
\title{Uniform estimation in stochastic block models is slow}
\runtitle{Uniform estimation in stochastic block models}

\begin{aug}
\author{\fnms{Isma\"el} \snm{Castillo}\ead[label=e1]{ismael.castillo@sorbonne-universite.fr}}
\address{
  Laboratoire Probabilit\'es, Statistique et Mod\'elisation\\
  Sorbonne Universit\'e\\
  \printead{e1}
}
\author{\fnms{Peter} \snm{Orbanz}\ead[label=e2]{porbanz@gatsby.ucl.ac.uk}}
\address{
  Gatsby Computational Neuroscience Unit\\
  University College London\\
  \printead{e2}
}

\runauthor{I. Castillo and P. Orbanz}

\affiliation{LPSM, Sorbonne Université and Gatsby Unit, UCL}

\end{aug}

\begin{abstract}
  We explicitly quantify the empirically observed phenomenon that
  estimation under a stochastic block model (SBM) is hard if the model
  contains classes that are similar. More precisely, we consider 
  estimation of certain functionals of random graphs  generated by a
  SBM. The SBM may or may not be sparse, and the number of classes
  may be fixed or grow with the number of vertices.
  Minimax lower and upper bounds of estimation along specific submodels are derived.  The results are nonasymptotic and imply that uniform estimation of a single connectivity parameter is much slower than the expected asymptotic pointwise rate. Specifically, the uniform quadratic rate does not scale as the number of edges, but only as the number of vertices. The lower bounds  are local around any possible SBM.
 An analogous result is derived for functionals of a class of smooth graphons.
\end{abstract}


\end{frontmatter}

\section{Introduction}

Network data occurs in a range of fields, and its analysis has become a highly interdisciplinary effort
\citep{Easley:Kleinberg:2010:1,Handcock:Gile:2010:1,Kolaczyk:2009:1,Newman:2009}.
In statistical network analysis, two classes of models have recently received particular attention: Graphon models
\citep{Borgs:Chayes:Lovasz:Sos:Vesztergombi:2008,Hoff:Raftery:Handcock:2002:1,Orbanz:Roy:2014}, 
and the subclass of stochastic block models (SBMs) \citep{allmanetal11,bickel13,Goldenberg:Zheng:Fienberg:Airoldi:2009}.
The results of this paper show, informally speaking, that estimation
under a SBM becomes difficult if the parameters specifying two
classes are close to each other.

SBM and graphon models parametrize a random graph by a 
symmetric measurable function $w$, which can be interpreted as representing an adjacency matrix in the limit
of infinite graph size \citep{Borgs:Chayes:Lovasz:Sos:Vesztergombi:2008}. 
In a SBM, the function is in particular piece-wise constant. 
Examples of statistical problems arising in this field include \emph{estimation problems} (see below), 
\emph{class label recovery}
\citep{bickelchen09,Lei:Rinaldo:2015:1,Lei:Zhu:2017:1,vvvp,zhangzhou}, and \emph{signal detection}, which refers to testing for the presence of a signal in settings where observed data constitutes a network or array \citep{accd11,ariascv14,butuceaingster,verzelenac15}. 

We consider estimation problems. SBMs label each vertex in a graph
with a category (a ``community''), and this labelling is typically not
observed. There is a substantial body of work on rates of estimation
in such models \citep{bickel13,allmanetal11,celisseetal,channarondetal,ambroisematias12,bickeletal11}. 
This literature considers asymptotic pointwise rates, and shows that,
informally speaking, estimators of finite-dimensional statistics can converge quickly even if the
labelling of vertices is not observed. Estimation of the
entire function $w$ has also been studied
\citep{chaogr15,kvt16,wolfeohlede}. In a case where this parameter has
infinite dimension and is estimated in a uniform
way, \citet*{chaogr15} show that not 
observing labels slows the rate.
Our results show that that is not a consequence of the nonparametric
setting: The best uniform (rather than pointwise) rate for estimating a finite-dimensional statistic---even a very
simple one, and under a very simple parametric SBM---is slow. The same
holds for a simple, one-dimensional functional of a smooth, infinite-dimensional
parameter function $w$.

\subsection{An informal overview}

The remainder of this section provides an informal overview of our
results. Rigorous definitions and statements follow in the next
sections. A  
SBM is defined by two parameters, a probability distribution
${\pi}$ on $k$ categories (which we regard as a vector in $[0,1]^k$), and a matrix ${M\in[0,1]^{k\times k}}$.  
The model generates undirected random graphs $G_n(\pi,M)$ of any size
${n\in\mathbb{N}}$: Each vertex ${i\leq n}$ is assigned a category
$\varphi(i)\in\{1,\ldots,k\}$ drawn randomly from $\pi$, and an edge between vertices $i$
and $j$ is then added with probability ${M_{\varphi(i)\varphi(j)}}$. (A
proper definition follows in Section \ref{sec:defs}.)
\\[.5em]
{\bf The main results}.
SBMs pose an estimation problem: Given an observed graph
$G_n$ with $n$ vertices, estimate $\pi$ and $M$.
The purpose of our work is to
show that, loosely speaking, the estimation problem can be harder than it appears,
or indeed than previous results may suggest at first glance.
Our results are phrased as minimax lower bounds: Suppose $\mathcal{S}$
is a set of parameter pairs $(\pi,M)$. A 
minimax bound specifies a decreasing function $\tau$ such that, informally,
\begin{align*}
  \inf_{\text{all estimators}}\;\sup_{(\pi,M)\in\mathcal{S}}\;\mathbb{E}_{\pi,M}[\|& \text{estimate
      of $(\pi,M)$ computed from $G_n(\pi,M)$} \\
    &  \;-\;(\pi,M)\|]^2
  \;\geq\;
  \tau(n)\;.
\end{align*}
That is, given an observed graph with $n$ vertices, there exists no estimator whose quadratic risk
is smaller than $\tau(n)$ for all
parameter values in $\mathcal{S}$.
Since the supremum means that shrinking
$\mathcal{S}$ will not increase the lower bound, it can suffice to
consider a subclass ${\mathcal{S}'\subset\mathcal{S}}$ of
parameters---any lower bound for $\mathcal{S}'$ is also a lower bound
for $\mathcal{S}$. Indeed, we will see that one can obtain a
meaningful bound by choosing a very small subclass with one degree of freedom, where
$\pi$ is fixed to the uniform distribution, and the matrix $M$ is a
function $M(\theta)$ of 
a one-dimensional surrogate parameter ${\theta\in[-1/2,1/2]}$.
The statement above then takes the form
\begin{align}
  \inf_{\text{all estimators}}\;\sup_{\theta\in[-1/2,1/2]}& \;\mathbb{E}_{\theta}[\text{estimate
      of $\theta$ computed from $G_n(\pi,M(\theta))$}\;-\;\theta]^2 \nonumber \\
  & \;\geq\;
  \tau(n)\;.\label{bound:intro}
\end{align}
Our main result shows that the relevant 
lower bound is 
\begin{equation*}
  \tau(n)\;=\;\frac{\text{constant}\cdot k}{n}\;,
\end{equation*}
where $k=k(n)$ is permitted to depend on $n$.
This is Theorem \ref{thm-twocl} (for the simplest case ${k=2}$) and 
Theorem \ref{thm-kcl} (for $k=k(n)\geq 2$).
Indeed our proofs imply a stronger statement:
\begin{itemize}
\item The results are completely non-asymptotic, and
  it is possible to explicitly determine numerical values for all relevant
  constants. See Remark \ref{remark:numerical:constants} for an example.
\item The minimax bound holds \emph{locally}, not just globally:
  In principle, a slow minimax rate may be caused by just a few ``pathological'' points in the set $\mathcal{S}$.
  One can ask whether the rate $\tau$ can be
  improved by removing a small part of $\mathcal{S}$. 
  That is not the case
  here: Shrinking
  the set of all SBM parameter pairs $(\pi,M)$ to any open Euclidean
  neighborhood of any specific pair still results in the
  same rate (see Section \ref{sec2:comments}).
  Informally, every
  region of parameter space contains parameters that prevent the rate
  from improving. 
\end{itemize}
SBMs are often used in ``sparse'' forms, and we verify
in Appendix \ref{sec_spa} that the result also applies in the
sparse case. Since sparsification reduces the amount of available
data, it slows convergence.
Theorems \ref{thm-twocl-spa} and \ref{thm-kcl-spa} show the rate in the sparse setting also scales
linearly in the expected number of edges.
\\[.5em]
{\noindent\bf Interpretation}.
If we were to simplify the estimation problem artificially by assuming that
the assignments variables $\varphi(i)$ are observed, $\pi$ could be estimated at rate ${1/\sqrt{n}}$ 
and $M$ at rate $1/n$, both by computing sample averages.
(The rates differ since $\pi$ is estimated from $n$ vertices, whereas $M$ is estimated from edges, and the expected number of edges grows quadratically 
with $n$.) 
Since our bounds are phrased in terms of a quadratic risk, both rates
must be squared: in the sequel,  a bound
${\tau(n)\approx 1/n}$ as above is referred to as  {\em slow rate}; by contrast,  a {\em fast rate} corresponds to  
${\tau(n)\approx 1/n^2}$. 

One must distinguish uniform rates (which hold uniformly over sets of parameters) and asymptotic pointwise rates
(where asymptotics in $n$ are considered at just a given point).
Previous work has established that estimation of $M$ can be fast
even if the $\varphi(i)$ are not observed. For example, a remarkable result of 
\citep{bickel13} shows that,
if $\hat{\pi}$ and $\hat{M}$ are chosen as certain profile
maximum likelihood estimators, then, as $n\to\infty$,
\begin{align} \label{asyn}
  \sqrt{n}(\hat{\pi}-\pi)
  \;&\rightarrow\;Z_\pi
  &&\text{and}&
  n(\hat{M}-M)
  \;&\rightarrow\;Z_M
  &&\text{in distribution}.
\end{align}
where $Z_\pi$ and $Z_M$ are multivariate Gaussian variables.
This holds up to label switching (see Section \ref{sec:defs}), and requires
that the columns of $M$ are ``not too similar''.
Related results can be found in \citep{allmanetal11,celisseetal,channarondetal}. 
Since this result does not use a quadratic risk, it can be paraphrased
informally as:
\begin{itemize}
  \item
  Under suitable conditions on the model, the matrix $M$ can be
  estimated, at least asymptotically and pointwise, at a fast rate.
\end{itemize}
It has long been recognized in statistics that pointwise asymptotic rates
can be hard to interpret:
As $(\pi,M)$ runs through some
set $\mathcal{S}$, the constants implicit in the rate may change locally around a given parameter as well as with $n$, and if they do so quickly
enough, that results in an effective change in the rate. The Hodges phenomenon, for example,
illustrates that highly pathological behavior of an estimator may only be visible in its uniform rate, whereas the asymptotic pointwise rate
suggests good performance \citep[see e.g. Section 8 and Figure 8.1 in][]{vanDerVaart}.
Our result says:
\begin{itemize}
  \item If measured uniformly over any given neighborhood in parameter
  space, the best achievable rate for connectivity parameters (i.e. for $M$) is always slow.
\end{itemize}
In other words, the change of constants is indeed an issue here, and makes the rate drop from a
fast to a slow one. Since the minimax bound is local, this problem
cannot be avoided by removing some (fixed) parameters $(\pi,M)$.\\[.5em]
{\bf Further results}. Section \ref{sec:ub} and the Appendix provide additional results on
achievability, i.e.\ upper bounds to complement the lower ones, and on graphon models and
sparse graphs.\\[.5em]
{\em Upper bounds}. Like most
lower bound results, Theorems \ref{thm-twocl} and \ref{thm-kcl} do not
show whether the bound $\tau$ is achievable---that is, the convergence
rate of any actual estimator could be even slower than $\tau$.
To show that a rate is achievable uniformly, one has to specify an estimator  whose uniform risk matches the lower bound up to constants.  Estimators for SBMs and their convergence rates are subject of a substantial literature,
but these rates are, once again, generally pointwise. Obtaining a
tight uniform upper bound for arbitrary SBMs is beyond the scope of
this work, but we do consider the ``hard'' one-dimensional model \eqref{bound:intro},
and show the following:
\begin{itemize}
  \item For estimation of $\theta$ in \eqref{bound:intro}, the rate
    $\tau$ is achieved by a type of maximum likelihood estimator. 
    That is shown in Theorem \ref{thmub} (for $k=2$) and
    in Theorem \ref{ubk} (for general $k$), in Section \ref{subsec-mle}.
\end{itemize}
However, this estimator is not generally computable in polynomial
time, which raises the additional question whether the problem exhibits
a computational gap---that
is, whether this is a problem where a sample of size $n$ contains
enough information to achieve a given rate $\tau(n)$, but this
information cannot be extracted in polynomial time, and every
practically computable estimate converges at a slower rate.
In this context, we show:
\begin{itemize}
  \item Under additional conditions, a spectral estimator
    (based on work of \citet{Lei:Zhu:2017:1}) achieves $\tau$, and is
    computable in polynomial time. (See Theorem \ref{thm_ubspec}
    in Section \ref{sec:spec:2} for $k=2$ classes, and the appendix for the general case
    $k\geq 2$.)
\end{itemize}
Thus, in the submodel specified by the conditions, there is no computational gap. We do not know
at present whether the same holds for general SBMs.\\[.5em]
{\em Smooth graphons}.
SBMs are a special case of so-called graphon models, which parametrize
a random graph $G_n(w)$ on $n$ vertices by a function $w$ of a certain
form. In SBMs, $w$ is piece-wise constant. Section \ref{sec4} instead considers
a class $\mathcal{S}$ of smooth graphons. It is known that uniform
estimation of such a graphon $w$ from $G_n(w)$ is only possible
at a slow rate \citep{chaogr15,kvt16}. Theorem \ref{thm-quad}
considers a simple, real-valued statistic
$\vartheta(w)$ that can be read as a form of standard deviation. It
shows that
\begin{align*}
  \inf_{\text{all estimators}}\;\sup_{w\in\mathcal{S}}& \;\mathbb{E}_{w}[\text{estimate
      of $\vartheta(w)$ computed from $G_n(w)$}\;-\;\vartheta(w)]^2\\
&  \;\geq\;
  \text{constant}\cdot\frac{1}{n}\;.
\end{align*}
In other words, even if the 
infinite-dimensional quantity $w$ is substituted by the much simpler,
one-dimensional quantity $\vartheta$, the rate is still slow.
In this sense, Theorem \ref{thm-quad} can be seen as a semiparametric counterpart to
the nonparametric results of \citep{chaogr15}.

\makeatletter
\renewcommand\tableofcontents{%
    \@starttoc{toc}%
}
\makeatother
\subsection{Contents}
\setcounter{tocdepth}{1}
\renewcommand{\contentsname}{\empty}
\tableofcontents

\vp

\section{Preliminaries and notation}
\label{sec:defs}

This section defines the models we consider, and briefly reviews some
related background.\\[.5em]
{\bf Notation}. We abbreviate ${[k]:=\{1,\ldots,k\}}$, so that  
${[k]^n}$ is the set
 of all mappings $\{1,\ldots,n\}\to \{1,\ldots,k\}$. 
For a subset $A$ of the integers, $|A|$ denotes cardinality. If $M$ is
a square matrix, $\|M\|_F$ is its Frobenius norm 
and $\|M\|_{Sp}$ its spectral norm. Let 
  Be$(p)$ be a shorthand for the  Bernoulli$(p)$ distribution. By $ER(p)$, we denote the law of
an Erd\"os-Renyi random graph edge probability $p$ over $n$ nodes,
that is, ${ER(p)=ER(p,n)=\otimes_{i<j\leq n} \text{Be}(p)}$, where $\otimes$ denotes a tensor product of distributions. 
\\[.5em]
{\bf Stochastic block models}.
Consider sampling at random 
an undirected, simple graph $G$ 
 on the vertex set $\cV=\{1,\ldots,n\}$ as follows.  Fix some $k\in\{1,\ldots,n\}$. 
Let ${\pi=(\pi_1,\ldots,\pi_k)}$ be a probability distribution 
on the set ${\lbrace 1,\ldots,k\rbrace}$, with $\pi$ identified as a line 
vector of size $k$. 
Let ${M}:=(M_{lm})$ be a symmetric ${k\times k}$ matrix with elements ${M_{lm}\in[0,1]}$.
To sample a graph $G$, we generate its adjacency matrix $X=(X_{ij})_{i,j\in\cV}$. Since $G$ is undirected, it suffices
to sample entries with ${i<j}$, 
\begin{enumerate}
  \item For each vertex $i\in \cV$, independently generate a label ${\vphi(i)\,\sim\,\pi}$.
  \item For each pair ${i<j}$ in $\cV$, independently sample from the distribution ${X_{ij}\,|\, \vphi(i),\vphi(j)\,\sim\, \text{Be}(M_{\vphi(i)\vphi(j)})}$.
\end{enumerate}
In this notation, $\vphi$ is a (random) mapping $\vphi:\{1,\ldots,n\}\to\{1,\ldots,k\}$ that attributes a label to each node of the graph. It is random because labels are by definition randomly sampled. 
The distribution $P_{\pi,M}$ so defined on the set of undirected, simple graphs
 is called a \emph{stochastic blockmodel} of order $k$ with parameters $\pi$ and $M$. 
One can also write 
 \begin{align}
 \begin{aligned} \label{modsbm}
 (\vphi(1),\ldots,\vphi(n))\ & \sim\ \pi^{\otimes n} \\
  (X_{ij})_{i<j}\, \given\, \vphi\ & \sim\  \bigotimes_{i<j}  \text{Be}(M_{\vphi(i)\vphi(j)}),
 \end{aligned} 
 \end{align}         
where  $\pi^{\otimes n}=\pi\otimes\cdots\otimes \pi$, and here and in the sequel $i<j$ refers to all pairs of indices
$(i,j)\in\cV^2$ with $i<j$. Any given $\vphi$ partitions the vertex set ${\lbrace 1,\ldots,n\rbrace}$ into
$k$ distinct classes.
We call $\pi$ the {\em proportions} vector and $M$ the matrix of {\em connectivity} parameters. 
\\[.5em]
{\bf Graphon models}. 
SBMs can be regarded as a special case of a more general class of 
random graphs, parametrized by
the set
of all measurable functions ${w:[0,1]^2\rightarrow[0,1]}$ that are symmetric,
i.e.\ ${w(x,y)=w(y,x)}$. Any such $w$ defines a random graph $G$: denoting  by $\text{Unif}[0,1]$ the uniform distribution on $[0,1]$, and $(U_i)_i=(U_i)_{1\le i\le n}$, set
 \begin{align}
  \begin{aligned} \label{modgr}
     (U_i)_{i} \ & \sim \ \text{Unif}[0,1]^{\otimes n} \\
     (X_{ij})_{i<j}\, \given\, (U_i)_i  \  & \sim  \ 
           \bigotimes_{i<j}  \text{Be}(w(U_i,U_j)).
   \end{aligned}
\end{align}     
The law $P_w$ of the graph $G$ defined by the random matrix $X$ in \eqref{modgr} is called a \emph{graphon model}
\citep{Borgs:Chayes:Lovasz:Sos:Vesztergombi:2008}.
SBMs are recovered by choosing $w$ as a histogram:
subdivide the unit interval into $k$ intervals ${I_s:=[\,\sum_{i<s}\pi_i,\sum_{i\leq s}\pi_i)}$
of respective lengths $\pi_s$, and set
\begin{equation}
  \label{SBM:graphon}
  w(x,y):=M_{ij} \quad\text{ for }\quad x\in I_i, y\in I_j\;.
\end{equation}
Then $P_w=P_{\pi,M}$. In a graphon model, the continuous vertex labels $U_i$
are almost surely distinct; in a stochastic block model, labels coincide whenever two 
vertices belong to the same class. Thus, the SBM labels 
can be regarded as discretization of graphon labels. Conversely, any graphon can be approximated
by a sequence of stochastic blockmodels of increasing order $k$; indeed, the set of stochastic
blockmodels---that is, of graphons of the form \eqref{SBM:graphon} for all $k$, $\pi$ and $M$---is
dense in the set of functions $w$ endowed with its natural topology \citep[see e.g.][for details]{Janson:2013:1}. This idea
can be used to construct SBM-valued estimators for graphons
\citep{wolfeohlede,chaogr15}.
SBMs and graphon models both generalize to directed graphs, 
by dropping the symmetry constraints on $\pi$ and $w$, and requiring only 
${i\neq j}$ rather than ${i<j}$; 
in the following, we consider only the undirected case.
\\[.5em]
{\bf Label switching and identifiability}.
The distribution \eqref{modgr} remains invariant if $w$ is replaced by $w\circ g$, for any measure-preserving transformation $g$ of $[0,1]$: $P_w=P_{\tilde{w}}$ for ${\tilde{w}(x,y)=w(g(x),g(y))}$. More generally, two graphons $w$ and $w'$ are considered equivalent if ${P_w=P_{w'}}$. The equivalence class $\left<w\right>$ of $w$ is called a graph limit. 
Similarly in \eqref{modsbm}, if $\sigma$ is a fixed arbitrary permutation 
of $\{1,\ldots,k\}$, with permutation matrix $\Sigma$, then $P_{\pi,M}=P_{\pi\Sigma,\Sigma M\Sigma^T}$. The parameters of the SBM can only be recovered up to label switching. We refer to  \cite{allmanetal11} and \cite{celisseetal} for detailed identifiability statements.
\\[.5em]
{\bf Fixed and random design.} In models \eqref{modsbm}-\eqref{modgr}, the latent variables, respectively $\vphi$ and $U$, are random. Sometimes, a slightly different version of the model is considered, where $\vphi$ and $U$ are still unobserved, but fixed, non-random quantities. For instance, under this setting \eqref{modsbm}  becomes
\[(X_{ij})_{i<j}\,  \sim\  \bigotimes_{i<j}  \text{Be}(M_{\vphi(i)\vphi(j)}),\]
for a given, unknown, $\vphi:\{1,\ldots,n\}\to\{1,\ldots,k\}$, and the data distribution is denoted $P_{\vphi,M}$. 
Such models will be referred to as {\em fixed design} SBM and {\em random 
design}  SBM respectively. The term SBM as used in the literature
typically refers to a random design.
Some theoretical arguments simplify in the fixed design case, for
which the data distribution is a product measure, rather than a
mixture of products measures. Most results below are obtained for both cases.
\\[.5em]
{\bf Mixture interpretation.}
The $n$-tuple $(U_i)$ in a graphon model, or, equivalenly, the mapping $\vphi$ in a SBM, are in general
not observed, and can hence be interpreted as latent variables. In other words, the distribution of 
the data $(X_{ij})_{i<j}$ is a {\em mixture}.
The mixture representation is useful
to relate fixed and random designs to each other. In the random design
case, we have
\begin{equation}\label{distrib}
 P_{\pi,M} = \sum_{\vphi\in [k]^n} \mu_{\pi}[\vphi] \,
\bigotimes_{i<j} \text{Be}(M_{\vphi(i)\vphi(j)}),
\end{equation}
where $\mu_{\pi}[\vphi] = \prod_{l=1}^k \pi_l^{N_l(\vphi)}$ and
$N_j(\vphi)=\sum_{i=1}^n \1_{\vphi(i)=j}$
is the number of times the label $j$ is present. 
In the fixed design model, the labels given through $\vphi$ are also
unobserved, but fixed, so that the distribution is $P_{\vphi,M}$ given by 
\[ P_{\vphi,M} = \bigotimes_{i<j} \text{Be}(M_{\vphi(i)\vphi(j)}).\]

\section{Main results: lower bounds} \label{sec:lbs}

In this section, we first construct in Section \ref{sec2} natural submodels of a SBM with $k=2$ along which 
the two classes become close, and derive an estimation lower bound in terms of the quadratic risk for the submodel parameter. We then consider in Section \ref{sec3} the more general setting of a SBM with $k$ classes `containing' the previously constructed difficult submodel and derive a minimax estimation lower bound in this setting, which is local around any possible SBM of this type, as  we discuss in more detail in Section \ref{sec2:comments}.

\subsection{The case $k=2$}
\label{sec2}

Consider the set of distributions
\begin{equation} \label{modk2}
\cM  = \left\{ P_\te:= P_{e,Q^\te},\quad \te\in[-1/2,1/2] \right\}\;,
\end{equation}
where $e$ and $Q^\te$ are given by 
\begin{align} 
 e\ \  =  & \quad \left[\frac12\, , \, \frac12\right] \label{prop2} \\
 Q^\theta  \ = & 
\begin{bmatrix}
     \frac12 + \theta & \frac12 - \theta\\
     \frac12 - \theta  & \frac12 + \theta
\end{bmatrix}. \label{simat}
\end{align}
The set $\cM$ is a $1$--dimensional submodel of the set of all SBMs with at most two classes. For $\te=0$ the matrix $Q^0$ is degenerate and the model is simply an Erd\"os-Reyni graph model with parameter $1/2$, that is all edges are independent and have a probability $1/2$ of being present. 
SBMs with connectivity matrices that---like $Q^\te$ above---specify
only two, one for intra-group and one for between-group
connections, are known as affiliation models \citep[e.g.][]{ambroisematias12,ariascv14,Mossel2016}.

\begin{thm} \label{thm-twocl}          
Consider a stochastic blockmodel \eqref{modsbm} with $k=2$ specified by 
$\cM$, that is $P_\te=P_{e,Q^\te}$ with $e,Q^\te$ given by  \eqref{prop2}-\eqref{simat}. There exists a constant $c_1>0$ such that  for all $n\ge 2$,
\[ \inf_{T} \sup_{\te\in[-1/2,1/2]} E_{\te} \left[ T(X) - \theta \right]^2       
\ge \frac{c_1}{n}, \]                                            
where the infimum is taken over all estimators $T$ of $\te$ in the model $\cM$. 
\end{thm}  
\begin{proof} See Section \ref{secpr2}.
\end{proof}

Theorem \ref{thm-twocl} states that, even in a very simple SBM with ${k=2}$ classes and only one unknown parameter in its connectivity matrix, the minimax estimation rate is no faster than $1/n$.
This is no contradiction to the fast rate obtained by \citet{bickel13} (meaning a $1/n$ rate for the convergence in distribution but a $1/n^2$ rate for the quadratic risk): the latter is a pointwise asymptotic result, and assumes that no two lines of the connectivity matrix are the same, whereas Theorem \ref{thm-twocl} is nonasymptotic and uniform. It shows that the rate in a two-class model changes for distributions close to an Erd\H{o}s-Renyi model ($k=1$); informally, models close to the `boundary' are harder to estimate. We note the result does not require the sub-model $\cM$ to include the Erd\H{o}s-Renyi model; see the remark below. 
The phenomenon is reminiscent of effects familiar from community detection, where matrices similar to  \eqref{simat} naturally arise as most difficult submodels. Community detection is a {\em testing} problem, though, as opposed to the {\em estimation} problem considered here.
For a different but related result in the very sparse case, see \citep{Mossel2016}.\\

\begin{remark}[different parameter choices] One can easily check that 
the result of Theorem \ref{thm-twocl} remains unchanged if  instead of
$1/2$ in the matrix $Q^\te$ in \eqref{simat}, another number
$a_0\in(0,1)$ is used. If $a_0$ is bounded away from $0$ and $1$,
assuming  $\min(a_0,1-a_0)\ge \rho>0$, then the result is only modified by constants.  Also, if the proportions vector $\pi$ is of the form 
$[b\,,\,1-b]$ with $b>0$, similar results continue to hold, provided the matrix $Q^\te$ is replaced by 
\[  
 Q^\theta_b  \ = 
\begin{bmatrix}
     \frac12 + c_b\theta & \frac12 - d_b\theta\\
     \frac12 - d_b\theta  & \frac12 + c_b\theta
\end{bmatrix}
  \]
for  suitable constants $c_b, d_b$ that depend on $b$ (one can take e.g. $c_b=1-b$ and $d_b=b$).
\end{remark}

\begin{remark}[numerical constants]
\label{remark:numerical:constants}
In Theorem \ref{thm-twocl}, one can take $c_1=1/107$; additionally, the 
supremum can be restricted to $(-\te_n,\te_n)$ for
\[ \te_n=\frac{c_0}{\sqrt{n}} \qquad \text{and}\quad c_0=\frac{1}{3\cdot2^{3/4}}\approx 0.56.\]
Moreover, the proof implies that one can restrict the supremum to a set not actually containing $\te=0$, but rather two points close enough to $\te=0$, namely $\te_1=c_1/\sqrt{n}$, $\te_2=c_2/\sqrt{n}$ for suitably chosen, fixed constants $c_1, c_2>0$.\\
\end{remark}
 
\noindent{\bf Fixed design}. A result similar to Theorem \ref{thm-twocl} holds for fixed designs.
In this case, the map $\vphi$ is deterministic, and the model can be written as
 $\cM_F=\{P_{\te,\vphi}:=P_{\vphi,Q^\te},\, 
\te\in[-1/2,1/2],\ \vphi\in [2]^n\}$. Expectations with respect to the measures $P_\te$ and $P_{\te,\vphi}$ are denoted respectively $E_\te$ and $E_{\te,\vphi}$. We then have
\[ \inf_{T_f} \sup_{\te\in[-1/2,1/2],\,\vphi\in[2]^n} 
E_{\te,\vphi} \left[ T_f(X) - \theta \right]^2
\ge \frac{c_1}{n}, \]  
where the infimum is taken over all estimators  $T_f$ of $\te$ in the fixed design model. The proof is the same as for Theorem \ref{thm-twocl}, see Section \ref{secpr2}.   \\

\subsection{The general case} 
\label{sec3}

We now consider an arbitrary number $k$ of classes. Above, we have perturbed a SBM with ${k=2}$ classes around an Erd\"os-Renyi model. We now similarly perturb a $k$-class SBM around one with ${k-1}$ classes.
The connectivity matrix of a SBM with at most ${k-1}$ classes 
is of the form, for $a_0,a_i,b_{ij}\in[0,1]$ for $i,j\in[k-2]$,
\begin{equation} \label{coM}
 M = 
\begin{bmatrix}
a_0 & a_1 & \cdots & a_{k-2} \\
a_1 & b_{11} & \cdots & b_{1 k-2}  \\
\vdots &   \vdots & & \vdots \\
a_{k-2} & b_{1 k-2} & \cdots & b_{k-2 k-2}                     
\end{bmatrix}.
\end{equation}
For simplicity of notation and easy comparison with Section \ref{sec2}, we assume $a_0=1/2$ throughout. Results are easily adapted to the case $a_0\in (0,1)$, requiring only that $a_0$ be bounded away from $0$ and $1$. If needed, one can ensure the number of classes is 
exactly ${k-1}$ by requiring no two rows of $M$ coincide, which we require only in Theorems \ref{thmspec} and \ref{thm-twocl-spa}, which describe the behavior of spectral estimators.

We consider $1$-dimensional submodels in the parameter space of connectivity matrices: Set 
\begin{equation} 
  e_k  = \left[\frac1k, \cdots, \frac1k \right], \label{propk} 
\end{equation}
and, for coefficients $\{a_i\}, \{b_{ij}\}$ as above, define
\begin{equation}
 M^\theta   
= \begin{bmatrix}
     \frac12 + \theta & \frac12 - \theta & a_1 & \cdots & a_{k-2} \vspace{.1cm}\\     
     \frac12 - \theta  & \frac12 + \theta & a_1 & \cdots & a_{k-2} \\
     a_1 & a_1 & b_{11} & \cdots & b_{1 k-2}  \\
\vdots &  \vdots & \vdots & & \vdots \\
a_{k-2} & a_{k-2} & b_{1 k-2} & \cdots & b_{k-2 k-2}    
\end{bmatrix}
=
 \begin{bmatrix}
Q^\theta & A\\
A^T & B                        
\end{bmatrix}
,
 \label{zte}
\end{equation}    
where
\[ Q^\theta  \ =  
\begin{bmatrix}
     \frac12 + \theta & \frac12 - \theta \vspace{.1cm}\\
     \frac12 - \theta  & \frac12 + \theta
\end{bmatrix} \]
and 
\[ A = 
\begin{bmatrix}
a_1 & a_2 & \ldots & a_{k-2} \\
a_1 & a_2 & \ldots & a_{k-2}
\end{bmatrix},  \quad    
B=  
\begin{bmatrix}
b_{11} & \cdots & b_{1 k-2}  \\
   \vdots & & \vdots \\
 b_{1 k-2} & \cdots & b_{k-2 k-2}       
\end{bmatrix}.
\]
Thus, $M^{\theta}$ is a symmetric $k\times k$ matrix, obtained from $M$ by replacing the scalar coefficient $a_0$ by the $2\times 2$ matrix $Q^{\theta}$, and repeating the vector $(a_i)_{1\le i\le k-2}$. \\[.5em]
The number of nodes in a given class will be specified as follows. 
For simplicity, we choose the proportions 
vector $\pi$ in \eqref{modsbm} equiproportional and equal to $e_k$ in \eqref{propk}. (As in the case $k=2$, analogous results can be obtained if the proportions are of similar sizes.)
Consider the model defined by
\begin{equation} \label{modk}
\cM_k  = \left\{ P_\te:= P_{e_k,M^\te},\quad \te\in[-1/2,1/2] \right\}\;,
\end{equation}
for $e_k, M^\te$ as in   \eqref{propk}-\eqref{zte}.
This is a $1$--dimensional submodel of the set of all SBMs with at most $k$ classes. For $\te=0$, the matrix $M^0$ again has two identical rows, and the model becomes a SBM with at most $k-1$ classes, with connectivity matrix given by $M$ defined in \eqref{coM}. 
By $E_\te$, we denote the expectation under $P_\te$ in the model $\cM_k$ given by \eqref{modk}.

\begin{thm} \label{thm-kcl}          
Consider a stochastic blockmodel \eqref{modsbm} with $k\ge 2$ classes specified by $\cM_k$ in \eqref{modk}, that is $P_\te=P_{e_k,M^\te}$ with $e_k,M^\te$ given by  \eqref{propk}-\eqref{zte}, for fixed matrices $A, B$ 
with arbitrary coefficients. 
There exists a  constant $c_3=c_3(\rho)>0$, independent of $A,B$,  such 
that,  for all $n\ge 12k$,
\[ \inf_{T} \sup_{\te\in[-1/2,1/2]} E_{\te} \left[ T(X) - \theta \right]^2       
\ge c_3\frac{k}{n}, \]                                            
where the infimum is taken over all estimators $T$ of $\te$ in the model $\cM_k$. 
\end{thm}  
\begin{proof}
  See Section \ref{secprk}.
\end{proof}

\noindent{\bf Fixed design}. A similar result holds for the fixed design case,  assuming that classes, given by the mapping $\vphi$,  are balanced in the following sense. Let  $\Sigma_e$ denote the set of maps $\vphi\in [k]^n$ such that, for some  constants $c_1,c_2$, for any $1\le j\le k$, 
\[ c_1\frac{n}{k} \le |\vphi^{-1}(j)| \le c_2\frac{n}{k}.\]
The set $\Sigma_e$ thus consists of those maps $\vphi$ that produce $k$ classes all of size of order $n/k$. Then the conclusion of Theorem \ref{thm-kcl} still holds, provided $E_\theta$ is replaced by $E_{\te,\vphi}$, and the supremum taken over $\theta\in[-1/2,1/2]$ and $\vphi\in\Sigma_e$ as defined just above.

\subsection{Some comments on the results}
\label{sec2:comments}

Theorem \ref{thm-kcl} establishes that the minimax estimation rate of $\te$ in  model \eqref{modk} is at best of the order $k/n$, uniformly over $k$ and $n$. An intuitive explanation for this particularly slow rate is as follows: the phenomenon observed for $k=2$ is still present but this time the part of the matrix $Z^\te$ containing information about $\te$ 
is smaller, as only of the order $2/k$ of the nodes will be assigned to classes $1$ or $2$, which are the elements of the connectivity matrix that 
depend on $\te$. 

An important point is that this lower bound is minimax {\em local} (as opposed to more commonly proved minimax global results) that is, not only does this slow rate occur around one specific least-favorable point in the 
parameter space, it does occur around {\em any} point. More precisely:
If we start with any $k\ge 2$, any proportions vector, and any
connectivity matrix $M$ as in \eqref{coM} with $k-1$ classes, there
exists at least one submodel around $M$, namely $\cM_k$ in
\eqref{modk}, such that estimation of a connectivity parameter in $M$
cannot be faster than $k/n$.
In Theorem \ref{thm-twocl}, the model given by $\te=0$ is an
Erd\"os-Renyi graph, which raises the question whether the slow rate
in Theorem \ref{thm-twocl} is a consequence of the distinguished
properties of the Erd\"os-Renyi model. This is not the case. Proving 
such a local bound makes the proof of Theorem \ref{thm-kcl} more involved 
in the random design case, as one has to quantify the $L^1$-distance between two mixtures of probability measures, instead of between one fixed measure and a mixture as is often the case in proving minimax global bounds.

It is interesting to compare the rate in Theorem \ref{thm-kcl} to the one 
that would be obtained if the labels were observed. If $k$ is fixed, Lemma 2 in \citet{bickel13}  
gives a quadratic rate of order $1/n^2$ for connectivity parameters when labels are observed. This result can be easily adapted to the case where $k$ possibly grows with $n$, say in an asymptotic setting with $n\to\infty$ and $k/n\to 0$, leading to a quadratic rate of order $(k/n)^2$. The {\em uniform} rate in Theorem \ref{thm-kcl} is the square-root of this rate and thus much 
slower.

\section{Further results: upper bounds and smooth graphons} \label{sec:ub}

In this Section, we complement our main results by upper bounds (under some conditions when $k\ge 3$) and results for certain smooth graphons, which can be seen as a continuous analogue of the results for the SBM parameter $\theta$.\\


We establish upper bounds that show that the lower bounds in the previous section can be matched for certain subsets of connectivity matrices. In the case of $k\ge 3$ classes, the conditions are  arguably somewhat restrictive and can probably be improved. However, since the lower bounds are proved to be local around any possible SBM containing two classes that are close, the rate, if not matched, can only become worse. As we show below, some conditions are in fact necessary. Indeed, we give  an example in Section \ref{subsec:nec} where the rate drops further, illustrating the difficulty of the estimation problem.  
 
\subsection{Upper bounds via maximum likelihood}
\label{subsec-mle}

Theorems \ref{thm-twocl} and \ref{thm-kcl} provide
lower bounds. There are  corresponding, matching upper-bound, which we
obtain next.\\[.5em]
{\bf The case ${k=2}$}. We define an estimator of $\theta$ as follows. For any $\sigma$ an element of $[2]^n$, i.e.\ for any mapping $\{1,\ldots,n\}\to\{1,2\}$, define
\begin{align} 
 2Z_n(\sigma,X) & := - \sum_{i<j,\ \sigma(i)=\sigma(j)} (1-2X_{ij}) +
\sum_{i<j,\ \sigma(i)\neq \sigma(j)} (1-2X_{ij}). \label{zedns}
\end{align}
Maximising \eqref{zedns} in $\sigma$ leads to set 
\[ \hat \sigma  = \underset{\sigma\in [2]^n}{\text{argmax }} |Z_n(\sigma,X)| \]
which leads to the profile maximum likelihood estimate
\begin{align}
\hat \te & = \frac{Z_n(\hat\sigma,X)}{b_n}\qquad\text{ and }\qquad  b_n 
 ={n\choose 2} = \frac{n(n-1)}{2}  \label{pseml}.
\end{align}
This estimator 
can be seen as  a (pseudo)-maximum likelihood estimate, see Appendix \ref{sec-ub}.   

\begin{thm} \label{thmub}
Consider a stochastic blockmodel \eqref{modsbm} with $k=2$ specified by 
$\cM$, that is, $P_\te=P_{e,Q^\te}$ with $e,Q^\te$ given by  \eqref{prop2}-\eqref{simat}. Let $\smash{\hat\te}$ be the estimator defined by \eqref{pseml}. There exists a constant $C_1>0$  such that for all $n\ge 2$, 
\[ \sup_{\te\in [-1/2,1/2]} E_{\te}\left[ \hat\te - \te \right]^2 
\le \frac{C_1}{n}.
\]
The same risk bound holds for $\hat\te$ in the fixed design model, uniformly over $\te$ and $\vphi\in [2]^n$.
\end{thm}    
\begin{proof}
  See Appendix \ref{sec-ub:2}.
\end{proof}
The main takeaway from this result is that the {\em uniform} quadratic rate for estimating the connectivity parameter along the submodel $\cM$  is exactly of order $n^{-1}$, up to constants. That follows from combining 
Theorems \ref{thm-twocl} and \ref{thmub}. 
This `slow' rate (as compared to the asymptotic pointwise quadratic rate $n^{-2}$ 
of \eqref{asyn}) arises even if all other parameters---here, the vector of proportions $\pi$---are assumed known. 
The submodel built for ${k=2}$ can be regarded as a local perturbation of an Erd\"os-Renyi graph model with connection probability $1/2$. The drop in the rate is already noteworthy, as the rate of estimation of $p$ for a ER$(p)$ model is of the order $n^{-2}$. \\[.5em]
{\bf The case ${k\geq 2}$}.
For this case, we make additional (but fairly mild) assumptions on the matrix $M$. These conditions are for simplicity of presentation and could,  in some cases, be improved. Our main purpose here is to show that, for `typical' matrices $A$ and $B$ in \eqref{zte}, the rate of estimation of $\theta$ in \eqref{zte} is indeed exactly of the order $k/n$. In Section \ref{subsec:nec} below, we show that at least {\em some} conditions on possible matrices $A,B$ are necessary: for certain unfavourable matrices, the rate drops below $k/n$.
As was the case for Theorem \ref{thm-twocl}, the result of Theorem \ref{thmub}
remains unchanged if the constant $1/2$ in $Q^\te$ is replaced by 
any $a_0\in(0,1)$.

We modify the criterion function \eqref{zedns} by restricting it to a given subset $S\subset\{1,2,\ldots,n\}$ of indices,
\begin{equation} \label{znew}
 2Z_n(\sigma,S,X)=- \sum_{i<j,\ i,j\in S,\, \sigma(i)=\sigma(j)} (1-2X_{ij}) +
\sum_{i<j,\ i,j\in S,\, \sigma(i)\neq \sigma(j)} (1-2X_{ij}). 
\end{equation}
To avoid technicalities, we maximize over a grid, which constitutes no loss of generality. To this end, 
define the regular grid $\Theta_n=\{i/(2n^2),\ i= -n^2,\ldots,n^2\}$ in $\Theta=[-1/2,1/2]$, and 
\begin{align} 
 (\tilde\sigma,\tilde \te)\ & :=\ \underset{\sigma\in\Sigma_e,\, \te\in\Theta_n}{\text{argmin }} 
\ \sum_{i<j} (X_{ij} - Z^\te_{\sigma(i)\sigma(j)})^2 \label{ub-pre} \\ 
\tilde{S}_I\ & := \ \tilde\sigma^{-1}(\{1,2\}). \label{sun-hat}
\end{align}
Equation \eqref{ub-pre} defines a global maximum-likelihood type estimator, which is then used to obtain an estimate $\tilde S_I$ of the set of nodes labelled $1$ or $2$. Given this estimate, one can apply the profile-type method already used in the case ${k=2}$:
For $\tilde S_I$ as in \eqref{sun-hat}, $\tilde n_k= {|\tilde S_I| \choose 2}$,  and $Z_n$ as in \eqref{znew}, set
\begin{align} 
\hat \sigma_I & = \underset{\sigma\in\Sigma_e}{\text{argmax }} |Z_n(\sigma,\tilde S_I,X)| \label{sigi} \\
\hat \te & = \frac{Z_n(\hat\sigma_I, \tilde S_I,X)}{\tilde n_k}. \label{thgen}
\end{align} 
We require the coefficient $a_0$ of the matrix $M$ in \eqref{coM} to be sufficiently distinct from the remaining entries: Let $\cC=\{a_{i}, b_{ij}, 1\le i,j\le k-2\}$ be the set of coefficients of the matrices $A$ and 
$B$ in \eqref{zte}, with $a_0=1/2$,
\begin{equation} \label{techm}
\min_{c\in\cC} \left\{ |c-a_0| \right\} \ge 2\kappa >0.
\end{equation} 
\vspace{-.1cm}

\begin{thm} \label{ubk} 
Consider a stochastic blockmodel \eqref{modsbm} with $k\ge 2$ classes specified by  $P_\te=P_{e_k,M^\te}$ with $e_k$ and $M^\te$ given by  \eqref{propk}-\eqref{zte}, for fixed matrices $A$ and $B$. Define $\smash{\hat\te}=\smash{\hat\te}(X)$ as in \eqref{thgen}. Suppose  \eqref{techm} holds and that, for some small enough $d$ and $\kappa$ as in \eqref{techm},
\begin{equation} \label{techc}
k^3 \log{k}\le d \kappa^4 n.
\end{equation}
Then there exists a universal constant $C_1>0$ such that for $n\ge 5$,
\[ \sup_{|\te|\le \kappa} E_{\te}\left[\hat\te - \te\right]^2 
\le C_1\frac{k}{n}.\]
The same risk bound holds for $\hat\te$ in the fixed design model, uniformly over $|\te|\le \kappa, \vphi\in \Sigma_e$.
\end{thm}
\begin{proof}
  See Appendix \ref{sec-ub:k}.
\end{proof}
Note $\kappa$ in \eqref{techc} may depend on $k$ and $n$, and may go
to zero in a framework where $k,n$ go to infinity. Below are two examples for the behaviour of $\kappa$. These examples illustrate that our conditions are indeed met in commonly encountered settings, in particular, with high probability, if $M$ is a random matrix and $k$ does not grow too rapidly with $n$.\\

\noi {\em Example 1 (well-separated block).} If $\kappa$ is a fixed positive constant e.g. $1/4$, then the submatrix $Q^\te$ is well separated from the other coefficients of the matrix $M$. The procedure above then correctly picks up a sensible approximation of the true set $\sigma^{-1}(\{1,2\})$ via $\tilde S_I$ and the rate $k/n$ is achieved, as long as $k$ does not grow faster than $n^{1/3}/\log{n}$, an already fairly important number of classes.
\\

\noi {\em Example 2 (randomly sampled matrix $M$).} Suppose that the symmetric matrix ${M=:(c_{ij})}$ in \eqref{coM} is a random matrix whose upper triangular entries are drawn i.i.d. with uniform distribution $\cU[0,1]$, except $c_{11}=1/2$. The distribution of $|c_{ij}-1/2|$ except for 
$i=j=1$ is then $\cU[0,1/2]$, and it is a standard fact that the first order statistic of a uniformly distributed sample of size $N$ is Beta$(1,N)$ distributed. That implies the random variable $2\min_{c_{ij}\in\cC} |c_{ij}-1/2|$ has law $\text{Beta}(1,k(k-1)/2-1)$. Therefore, $\kappa$ in \eqref{techm} is of order no less than $1/k^2$ with high probability. From \eqref{techc} one deduces that for $k$ of the form $n^\delta$ with 
$\delta<1/11$ and $n$ large enough, the rate $k/n$ is achieved uniformly and locally, for typical matrices $M$. 
Inspection of the proof of Theorem \ref{ubk} reveals that $k=o(n^\delta)$ with $\delta<1/7$ in fact suffices for the rate $k/n$ to be attained with high probability when $M$ is random: this is achieved by distinguishing  $c_{ij}$ of the types $a_i$ or $b_{ij}$ in the proof and noting that the minimum of $|a_{i}-1/2|$ over $i$ will be of larger order $k^{-1}$, instead of $k^{-2}$ for the minimum over $i,j$ of $|b_{ij}-1/2|$.\\

\begin{remark}[conditions $|\theta|\le \kappa$ and \eqref{techm}]  \label{remcondsep}
We slightly restrict the range of $\theta$ in the upper bound of Theorem \ref{ubk}. Formally, the matching upper bound is obtained for a somewhat smaller interval than $[-1/2,1/2]$  when $k\ge 3$. (If $k$ is fixed and $n\to\infty$, the restriction is only to $[-\delta,\delta]$ for a small enough constant $\delta>0$.)  
The condition is needed to ensure, in combination with \eqref{techm}, that the block $Q^\theta$ in the matrix \eqref{zte} is separated sufficiently from the other submatrices $A$ and $B$. If this is not the case, the estimation problem can become more difficult, and the rate hence slower. This is formally shown in Section \ref{subsec:nec}, where the extreme case of all coefficients of $A,B$ being equal to $1/2$ is discussed. This phenomenon can also occur if only some parts of $A$ and $B$ are close to $1/2$, or to either $1/2+\theta$ or $1/2-\theta$ for some $\theta\in(0,1/2)$. 

We do not claim that the restriction to $[-\kappa,\kappa]$  and \eqref{techm} are sharp conditions; they can probably be improved. However, the argument above shows some condition of this form is needed, although it may vary depending on the estimation procedure considered: for spectral estimators as  considered in Appendix A, for example, we need a similar separation assumption, although it takes a slightly different form (see Theorem 7 in Apprendix A and the comments below it). We also note that small values of $\theta$ are conceptually the most interesting case, since the $2\times 2$ subproblem becomes easier the larger $\theta$ becomes. 

\end{remark} 

\subsection{Upper bounds via spectral estimates}
\label{sec:spec:2}

Since the maximum likelihood estimator \eqref{ub-pre}
has to optimize over the set $[k]^n$, it need not be computable in
polynomial time.
It hence seems natural to ask whether there is a ``computational
gap'', that is, whether the best estimator computable in polynomial
time converges at a slower rate than predicted by the minimax bound. We do not have
a complete answer to his question, but for a somewhat restricted model
class, no such gap exists:
The estimator described below for the case ${k=2}$ uses a spectral method, see e.g. \cite{Lei:Rinaldo:2015:1}.
A generalization to ${k\geq 3}$ classes is discussed in
Appendix \ref{sec:spec:k}, which requires further conditions.
Within the remit of these conditions, however, the minimax rate is
achievable in polynomial time. An extension to sparse graphs is considered in 
 Appendix \ref{sec_spa}. A small simulation study in Section \ref{sec:simulation} illustrates the behaviour of the estimator.

With the convention that $X_{ii}=1/2$ and $X_{ji}=X_{ij}$, define the 
$n\times n$ matrix $\Delta$ by
\[ 
\Delta := X - \frac12 J, \quad\text{ where }\quad
J := \bigl(1\bigr)_{i,j\leq n}. 
\]
Let $\la_{1}^a(\Delta)$ denote the largest eigenvalue in absolute value of $\Delta$ and set 
\begin{equation} \label{te_spec} 
\tilde\te := \frac{\la_{1}^a(\Delta)}{n-1}\;. 
\end{equation} 
We refer to this procedure as spectral algorithm for $k=2$ and denote it $\cS_2$.
The intuition behind this estimator in the fixed design setting is the following. For $i\neq j$, we have
\[ E[X_{ij} - \frac12] = M_{\vphi(i)\vphi(j)}^\theta - \frac12
=(-1)^{\1_{\vphi(i)\neq \vphi(j)}}\theta. \]
Set ${v=((-1)^{\mathds{1}\lbrace \vphi(i)=1\rbrace})_{i\leq n}}$ and $V :=  vv^t=\bigl((-1)^{\mathds{1}\lbrace \vphi(i)\neq\vphi(j)\rbrace}\bigr)_{i,j\leq n}$.
Then for non-random $\vphi$,
\[ E[\Delta] = \theta(V-I_n),\]
where $I_n$ is the identity matrix of size $n$. As $E[\Delta]$ is a
rank $1$ matrix whose non-zero eigenvalue equals $(n-1)\te$ (with $v$
the corresponding eigenvector), this leads us to introduce $\tilde\te$
as in \eqref{te_spec}.
\begin{thm} \label{thm_ubspec}
In the same setting as in Theorem \ref{thmub}, let $\tilde\te$ be the estimator defined by \eqref{te_spec}. There exists a constant $C>0$  such that for all $n\ge 2$, 
\[ \sup_{\te\in [-1/2,1/2]} E_{\te}\left[ \tilde\te - \te \right]^2 
\le \frac{C}{n}.
\]
The same risk bound holds for $\hat\te$ in the fixed design model, uniformly over $\te$ and $\vphi\in [2]^n$.
\end{thm}    
\begin{proof}
  This follows as a special case of Theorem \ref{thm-twocl-spa}, in
  Appendix \ref{sec_spa}.
\end{proof}

\subsection{Necessity of conditions on $M$} \label{subsec:nec}

What precedes shows that the rate $k/n$ is achieved under conditions on  $M$ in \eqref{coM} and/or $k$. In general, we expect the rate to depend on the matrices $M$. Although we do not investigate this point in full here, we discuss it briefly.

The estimation methods investigated in Section \ref{subsec-mle} (MLE) and Appendix \ref{sec:spec:k} (spectral method) require the upper-left ${2\times 2}$ block of $M^\te$ to be  sufficiently  separated from at least part of the other entries of $M^\te$. Among those matrices $M^\te$ whose upper-left corner equals $Q^\te$, a worst case scenario should correspond to a matrix whose coefficients in 
$A$ and $B$ all equal $1/2$. This leads to the matrix 
\begin{equation} \label{special}
 \check M^\theta   
= \begin{bmatrix}
     \frac12 + \theta & \frac12 - \theta  &  \frac12 &  \cdots & \frac12\\
     \frac12 - \theta  & \frac12 + \theta  & \frac12  &  \cdots & \frac12\\
     \frac12   &  \frac12   & \frac12   & \cdots & \frac12  \\
\vdots &  \vdots & \vdots & & \vdots \\
\frac12 & \frac12 & \frac12 & \cdots & \frac12    
\end{bmatrix},
\end{equation}
which is of course heavily over-specified from the SBM perspective.
Consider the  SBM  in a fixed design case, where $\vphi:\{1,\ldots,n\}\to 
\{1,\ldots,k \}$ is unobserved. Suppose all classes $\sigma^{-1}(i)$ are of cardinality of order $n/k$, and the  connectivity matrix is given by \eqref{special}. This specific model can be regarded as a special case of the setting considered from a testing perspective by \citet{butuceaingster} and \citet{ariascv14}.  From Theorem 4.3 of \cite{butuceaingster}, one 
can deduce that the minimax rate for the quadratic risk when estimating $\te$ is no better than $\rho_n=\min\left(\frac{k^2}{n},\sqrt{\frac{k\log k}{n}}\right)$, for $k,n\to \infty$ and $\rho_n=o(1)$. The rate is therefore no better than $k^2/n$ for $k\le n^{1/3}$, and remains much slower than $k/n$ even for $k>n^{1/3}$.

\subsection{Minimax rates for a class of functionals of smooth graphons}\label{sec4}

Stochastic block models can be identified with piecewise constant graphons; we now consider the case where
the graphon is a smooth function instead. 
Let $w:[0,1]^2\to [0,1]$
a measurable function, let $\left< w\right>$ be its graphon equivalence class, and denote by
${P_{\left< w\right>}=P_w}$ the distribution of data $X$ generated 
by the graphon model \eqref{modgr}. 
Consider the problem of estimating the functional
\begin{equation}\label{sig}
 \ta(\left< w\right>)=\left[\int_{[0,1]^2} \Big(w(x,y)-\int_{[0,1]^2}w\Big)^2 dxdy\right]^{1/2},
\end{equation} 
for any representer $w$ of $\left< w\right>$. 
This is well defined in terms of the graphon, as the integral is invariant under any simultaneous (Lebesgue-)measure-preserving transformation of $x$ and $y$.

The statistic \eqref{sig} can be interpreted as a `graphon-standard deviation'. 
Its estimation under a smooth graphon model is, in a sense, analogous to the problem of estimating the functional $\te$ in the simple SBM with two 
classes discussed in Section \ref{sec2}: Let $h_\te$ be the piece-wise constant graphon characterizing the SBM defined by \eqref{prop2}--\eqref{simat}. 
Since $\ta(\psg h_\te \psd)=|\te|$, estimating $\te$ is then indeed equivalent (for positive values) to estimation of $\ta(\left< h_\theta \right>)$. 

Under a 2-class SBM, the results of Section \ref{sec2} show $\te$ in \eqref{simat} cannot be estimated faster than $c/n$. It is natural to ask whether the same still holds if one works with `smoother' graphons instead of histograms (where we refer to $\psg w \psd$ as smooth if at least one of its representers is a smooth function). The following result addresses this question for a simple class of smooth graphons, both for $\ta(\cdot)$ and for a larger class of functionals containing $\ta(\cdot)$.

Let $\cP_B$ be the collection of all graphons that admit a representer which is a polynomial in $x, y$, with degree bounded by some integer $D\ge 2$ and coefficients bounded by an arbitrary constant $M>0$ (this boundedness restriction is only to ensure a --nearly, up to a log term-- matching 
upper-bound in the next result).  For any $0\le \theta\le 1$, let us denote by $w_\te$ the function from $[0,1]^2$ to $[0,1]$ given by
\begin{equation} \label{polyw}
 w_\theta(x,y)=\frac12-\theta(x-\frac12)(y-\frac12)
\end{equation} 
and let $w_0$ denote the  constant function equal to $1/2$. The function $w_\theta$ can be interpreted as a `smooth' counterpart to the histogram graphon underlying the SBM \eqref{simat}.

\begin{thm} \label{thm-quad}
Let $X$ be data from the graphon model \eqref{modgr}. Let
$\ta(\cdot)$ be defined as in \eqref{sig}. There exist constants $c_1,c_2>0$ such that 
\[
 \frac{c_1}{n}\le  \inf_{\hat \ta} \sup_{w \in \cP_B} E_{P_{w}}\left[\hat{\ta}(X) - \ta(\left< w\right>)\right]^2 \le \frac{c_2\log{n}}{n}, 
\]
where the infimum is taken over all possible estimators of $\ta(\psg w\psd)$ in model \eqref{modgr}. 
Let $\psi$ be an arbitrary functional defined on graphon equivalence classes satisfying 
\begin{equation} \label{condfun}
 |\psi(\psg w_\te \psd)-\psi(\psg w_0\psd)| \ge c|\te|
\end{equation}  
for some $c>0$, for any $0\le \te\le 1$, and for the function $w_\te$ in \eqref{polyw}, 
Then for some $d>0$,
\[ \inf_{\hat \psi} \sup_{w \in \cP_B} E_{P_{w}}\left[\hat{\psi}(X) - \psi(\left< w\right>)\right]^2 \ge \frac{d}{n}.
\]
\end{thm}
\begin{proof} See Section \ref{sec:lbg}.
\end{proof}

The first part of Theorem \ref{thm-quad} asserts that the quadratic minimax rate for estimating \eqref{sig} cannot be faster than $c/n$, even if one 
 restricts the parameter set to a small class of smooth graphons $w$, namely graphons with a polynomial representer of bounded degree. This class can be seen as a smooth analogue of the histogram graphon underlying model \eqref{modk2}, or more generally the model with $k$ classes and connectivity \eqref{zte}. The degree of the polynomial can be seen as the analog of $k$. The rate of order $1/n$ is obtained because the degree of the polynomial is assumed bounded. Although we do not investigate this further here, one may conjecture that the rate would slow even further for a larger class (e.g. growing degree of polynomials, or a nonparametric class such as a  H\"older ball). 
 
The second part of Theorem \ref{thm-quad} indicates that the specific form of the functional $\ta(\cdot)$ in \eqref{sig} is not essential for the lower bound to hold. A given functional $\psi(\cdot)$ leads to a rate at least as slow of $\ta(\cdot)$ over the considered class of graphons as soon as \eqref{condfun} holds. This condition intuitively means that the functional $\psi(\cdot)$ is at least as hard as to estimate as the functional $\ta(\cdot)$, for which the difference on the left hand-side of \eqref{condfun} indeed behaves like $|\theta|$.  By direct computation we see that an example of such a graphon functional is
\[ \psi(P_w) =  \int_{[0,1]^2} \Big|f(x,y) - \int_{[0,1]^2} |f(x,y)|dxdy \Big|dxdy.\]
Providing a unified theory with matching lower and upper bounds for graphon functionals is an interesting topic for future research.

\section{Discussion}  \label{secdisc}

\citet{chaogr15} show that, if one estimates the parameter function $w$ of a graphon model, not observing the vertex labels---in this case, the variables $U_i$ in 
\eqref{modgr}---does (in general) impact on the optimal rate. 
In the present paper,  we have considered uniform estimation of certain functionals of graphon models  (in particular, the loss function is quite different from theirs). 
For estimation of certain random graph functionals---including the connectivity parameters  considered by \citet{bickel13}---we have shown that
the uniform, minimax rate does depend on whether the labels are observed, 
i.e.\ the phenomenon described by
\cite{chaogr15} persists even if one does not try to recover the entire function $w$, but only a specific $1$--dimensional aspect of $w$. The fast quadratic rate  $1/n^2$ is not achievable uniformly. If the number $k$ of classes is known and fixed, the quadratic rate becomes $1/n$. If the number of classes $k$ grows with $n$, the rate drops 
to $k/n$. We have used some mild assumptions on the part of the connectivity matrix other than the $2\times 2$ submodel. If those assumptions are not satisfied, the rate may even drop further. Similar results also hold for sparse graphs.

Interestingly, for the functionals considered here, the uniform rate is always, regardless of the number of classes $k$, much below the rate in the case where labels would be observed. This is in contrast with the problem of recovery of the mean adjacency matrix considered in \cite{chaogr15}, where for $k$ is larger than $\sqrt{n\log{n}}$, the (non--normalised) rate $k^2+n\log{k}$ is dominated by the `parametric' rate $k^2$, the rate if labels are observed. 


We claim no novelty regarding the algorithms---the MLE and spectral method---which we have adapted from existing work to the problem at hand. Their purpose is to verify that the lower bound is tight (both algorithms achieve it) under some mild conditions, and that there is no computational gap (the spectral method does so in polynomial time). 
Yet, we are not aware of other work providing uniform rates for SBM connectivity parameters for these or other algorithms, which constitutes another novelty of the paper.


Aspects of our proofs  reflect the fact that graphon models constitute a specific type of mixture model,
and estimation in mixtures can be difficult if mixture components are hard to distinguish; although no general theory of these phenomena seems to exist, we refer to the early work on estimation in finite mixture models by \cite{Hartigan83} and \cite{bickelchernoff}, and e.g.$\,$ to \cite{kahn} and \cite{gmc16} for more recent results.

\section{Proofs of the lower bounds in SBMs} \label{secpr}

The proofs of Theorems \ref{thm-twocl} and \ref{thm-kcl} rely on variations of Le Cam's `two-points' method, which bounds the minimax risk from below by a quantity involving the $L_1$ distance between a 
distribution and a finite mixture. (Specifically, this is the `point versus mixture' variant of the two points method, see e.g. \cite{binyu}.) 
This and other relevant technical lemmas are recalled in Section  \ref{sec:lemmas} below; the two points method is 
Lemma \ref{lem-lc}. For Theorem \ref{thm-kcl}, for $k\geq 2$ classes, one main idea is to `isolate' the part corresponding to the submatrix $Q^\theta$. More details comments are given along the proof in Section \ref{secprk} below.\\

{\em Notation.} Recall that a SBM with $k$ classes, proportions vector $\pi$ and connectivity matrix $M$ has distribution $P_{\pi,M}$ as given in \eqref{distrib}.
For a $n\times n$ symmetric matrix $A$ with zero diagonal, we write
\[ P_A = \bigotimes_{i<j} \text{Be}(A_{i,j}). \]
If $A$ is only given by $A_{i,j}$ for $i<j$, one extends it by symmetry and sets $A_{i,i}=0$.
The distribution of a SBM in the fixed design case with given $k,M$ and labelling function $\vphi$  is hence $P_{A^\vphi}$, where $A^\vphi_{i,j}=A^\vphi(M)_{i,j}=M_{\vphi(i)\vphi(j)}$. In the random design case, 
if $\pi$ is the vector with equal proportions $e_k=[k^{-1},\ldots,k^{-1}]$, then from \eqref{distrib}, 
\[ P_{\pi,M}=P_{e_k,M}= \frac{1}{k^n} \sum_{\vphi\in [k]^n} P_{A^\vphi}\qquad\text{ where }\quad
A^\vphi_{i,j}=M_{\vphi(i)\vphi(j)}. \]
We generically denote universal constants by $C$, where the
value may change from line to line. \\

\subsection{Two classes} \label{secpr2}

\begin{proof}[Proof of Theorem \ref{thm-twocl}] \label{pr-thm-twocl}
Let $N=2^n$ and let $A_1,\ldots,A_N$ be the collection of symmetric $n\times n$ matrices with general term 
$a_{ij}(\vphi)=a_{ij}(\te,\vphi)=
Q_{\vphi(i)\vphi(j)}^\theta$, $i<j$ and zero diagonal, for all possible $\vphi\in [2]^n$ and some $\te\in\Theta$. Let $A_0$ be the $n\times n$ matrix with all elements equal to $1/2$ on the off-diagonal, that is the matrix with $\te=0$. By Lemma \ref{lem-lc}, applied with $\ta=0$ and $\theta=\theta_n$  small to be chosen below, in order to get a lower bound 
for the minimax risk, it is enough to bound the $L^1$-distance $\|\mathbb 
P -\mathbb Q\|_1$ between 
\[ \mathbb P = P_{A_0},\quad \mathbb Q = \frac1N \sum_{k=1}^N P_{A_k}. \] 
If $\la_1=\text{Be}(q)$, $\la_2=\text{Be}(r)$, $\mu=\text{Be}(s)$, a simple computation leads to
\[ \int (\frac{d\la_1}{d\mu}-1)(\frac{d\la_2}{d\mu}-1)d\mu 
= \frac{(q-s)(r-s)}{s(1-s)}. \]
By Lemma \ref{lem-tv} applied to $\mathbb P$ and $\mathbb Q$, where $\ta_{i,j}(\vphi,\psi)= (2a_{ij}(\vphi) - 1)(2a_{ij}(\psi) - 1)$,
\begin{align*}
\| \mathbb P - \mathbb Q\|_1^2 
& \le \frac{1}{4^n} \sum_{\vphi,\psi\in [2]^n} 
\prod_{i<j} (1+\ta_{i,j}(\vphi,\psi)) - 1 \\
& \le \left[\frac{1}{4^n} \sum_{\vphi,\psi\in [2]^n} 
e^{\sum_{i<j} \ta_{i,j}(\vphi,\psi)}\right] - 1,
\end{align*}
Note that 
$2a_{ij}(\vphi)-1=2\theta(-1)^{\1_{\vphi(i)\neq \vphi(j)}}$ for any $\vphi\in [2]^n$. Denote $\eta_i=\1_{\vphi(i)=1}-\1_{\vphi(i)=2}$ and $\eta_i'=\1_{\psi(i)=1}-\1_{\psi(i)=2}$, for any index $i$. We have
$2a_{ij}(\vphi) - 1 = 2\theta\eta_i \eta_j$, so that $\ta_{i,j}(\vphi,\psi)=4\theta^2\eta_i\eta_j\eta_i'\eta_j'$. The term under brackets in the last display can be 
interpreted as an expectation over $\vphi, \psi$, where both variables are sampled uniformly from the set of all mappings from $\{1,\ldots,n\}$ to 
$\{1,2\}$. 
Under this distribution, the variables $\eta_i$ for $i=1,\ldots,n$ are independent Rademacher, as well as the variables $\eta_i'$, and both samples are independent. Further note that the variables $R_i:=\eta_i\eta_i'$ for $i=1,\ldots n$ form again a sample of independent Rademacher  variables. It is thus enough to bound, 
\[ E \left[ e^{4\theta_n^2\sum_{i<j} R_i R_j}\right], \]
where $E$ denotes expectation under the law of the $R_i$.
The previous exponent is an instance of Rademacher chaos; its Laplace transform can be bounded using Lemma \ref{lemcaos}. If $Z_n:=\sum_{i<j} R_i R_j$, we have that for any $\veps$ (say $\veps=1/2$), there exists $\lambda>0$ such that for all $n\ge 2$,
\[ E  \left[e^{|Z_n|/(\lambda n)}\right] \le 1+\veps. \]
Choosing $n\theta_n^2:= 1/(4\lambda)$ leads to $ \|\mathbb P-\mathbb Q\|_1^2 \le \veps=1/2$, so that the minimax risk 
is bounded below by $(32n\lambda)^{-1}$. 

To obtain the constants as in the remark below the Theorem, using Lemma \ref{lemcaosc} in the final step of the proof with $\te_n^2=1/(12s_n)$, $s_n^2=n(n-1)/2$, $r(\cdot)$ as in Lemma \ref{lemcaosc}, gives
\begin{align*} 
 R_M & \ge \frac{\te_n^2}{4}\left\{1-\frac{1}{2}\sqrt{r(4\te_n^2 s_n)}\right\}, \\
& \ge  \frac{\sqrt{2}}{48 n} \left\{1-\frac{1}{2}\sqrt{r\left(\frac{1}{3}\right)}\right\}\ge \frac{0.45}{48 n}\ge\frac{1}{107n}. \qquad  \qedhere
\end{align*}   
\end{proof} 

\subsection{Lower bounds for $k$ classes} \label{secprk}

Here the problem is more delicate compared to $k=2$, as the typical number of nodes per class now depends on $k$, and, in the random design case, the data distribution for $\te=0$, around which we build the lower bound, is itself a mixture.
As a first step, we start by establishing a result in a fixed design setting, that is 
\begin{equation} \label{lbkfixed}
 \inf_{T} \sup_{\te\in\Theta,\, \vphi\in [k]^n} E_{\te,\vphi}\left[ T -\te \right]^2 \ge c\frac{k}{n}
 \qquad\text{ for some }c>0\text{ and }k\ge 3\;.
\end{equation}  
\begin{proof}[Proof of \eqref{lbkfixed}]
Define $m=m_k=2\lfloor \frac{n}{k} \rfloor$. 
Set $S_1=\{1,\ldots,m\}$ and $S_2=\{m+1,\ldots,n\}$.
Let $\vphi_0\in [k]^n$ be a mapping such that
\begin{equation} \label{lbtech1}  
 \vphi_0(S_1)\subset \{1,2\} \qquad\text{ and }\qquad \vphi_0(S_2)\subset 
\{3,\ldots,k\}.
\end{equation}                           
Let $\vphi\in [k]^n$ be such that 
\begin{equation} \label{lbtech2}    
 \vphi(i)=\vphi_0(i)\quad\text{ whenever }i\in S_2 \qquad\text{and}\qquad \vphi(S_1) \subset \{1,2\},
\end{equation} 
and denote by $\cF=\cF(\vphi_0,S_1)$ the set of all such $\vphi$'s.
Then the restriction $\vphi_{| S_1}=:\vphi_1$ of $\vphi\in\cF$ to $S_1$ 
can be identified to an element of $[2]^m$.

Let $M^\theta$ be the $k\times k$ matrix  defined in \eqref{zte}.
For $\vphi\in\cF$, let $R_{\vphi}$ denote the matrix with general term $r_{ij}=r_{ij}(\vphi)$ equal to $M^\theta_{\vphi(i)\vphi(j)}$.
There are as many such matrices as possible  $\vphi_1$s, that is $|[2]^m|=2^m$. As $\vphi$ and $\vphi_0$ are identical by construction on $S_2$,
\[      
 r_{ij}  =   
 \begin{cases}     
 M^\theta_{\vphi_1(i)\vphi_1(j)} & \mbox{if }(i,j)\in S_1\times S_1\\
 M^\theta_{\vphi_1(i)\vphi_0(j)} = a_{\vphi_0(j)-2} = M^0_{\vphi_0(i)\vphi_0(j)}
 & \mbox{if } (i,j)\in S_1\times S_2\\
  M^\theta_{\vphi_0(i)\vphi_1(j)} = a_{\vphi_0(i)-2}  = M^0_{\vphi_0(i)\vphi_0(j)}
 & \mbox{if } (i,j)\in S_2\times S_1\\
 M^0_{\vphi_0(i)\vphi_0(j)}  & \mbox{if } (i,j) \in S_2 \times S_2
 \end{cases}       
\]
where $\vphi_1$ belongs to $[2]^m$. 
Next set, with $A_0$   the matrix with general term $m_{ij}(\vphi_0)=M^0_{\vphi_0(i)\vphi_0(j)}$,
\[ \mathbb{P}'= P_{A_0}\qquad\text{ and }\qquad  \mathbb{Q}'= \frac{1}{2^m} \sum_{\vphi\in\cF}
P_{R_{\vphi}}.  \]
Now we apply Lemma \ref{lem-tv} to $\mathbb{P}', \mathbb{Q}'$. Both $P_{A_0}$  and $P_{R_\vphi}$ are product measures over all pairs of indices $(i,j)$ with $1\le i<j\le n$. By construction, the individual components of 
these products coincide as soon as either $i$ or $j$ does not belong to $S_1$. We write
\[ P_{A_0} = \bigotimes_{i<j} P_{A_0}(i,j)
\qquad\text{ and }\qquad
P_{R_\vphi} =    \bigotimes_{i<j} P_{R_\vphi}(i,j).
 \]
where, for any indices $i,j$ with $i<j$,
\[ P_{A_0}(i,j) = \text{Be}(M^0_{\vphi_0(i)\vphi_0(j)})\qquad\text{ and 
}\qquad
P_{R_\vphi}(i,j)=  \text{Be}(M^\te_{\vphi(i)\vphi(j)}). \]
For ${\vphi,\psi\in[4]^n}$, we set 
\[ \tau_{i,j}(\vphi,\psi)=  P_{A_0}(i,j)\left[\left(\frac{dP_{R_\vphi}(i,j)}{dP_{A_0}(i,j)}-1\right)
\left(\frac{dP_{R_\psi}(i,j)}{dP_{A_0}(i,j)}-1\right)\right]. \]
If $i$ or $j$ belongs to $S_2$, then 
$r_{ij}(\vphi)=M^0_{\vphi_0(i)\vphi_0(j)}=A_0(i,j)=r_{ij}(\psi)$ by 
definition,
in which case the last display equals $0$. 
In Lemma \ref{lem-tv}, where $\vphi,\psi$ play the role of the indices $k,l$.
Identifying $\psi_{|S_1}$ with the corresponding mapping $\psi_1\in [2]^m$, we have $\|\mathbb{P}'-\mathbb{Q}'\|_1^2\le \chi^2(\mathbb{Q}',\mathbb{P}')$ and 
\begin{equation}\label{fixed-fbound}
 \chi^2(\mathbb{Q}',\mathbb{P}') \le \frac{1}{2^{2m}} 
\sum_{ \vphi_1,\psi_1\in [2]^m }
\prod_{1\le i<j\le m} (1+\tau_{i,j}(\vphi_1,\psi_1)) - 1.
\end{equation}
The last expression coincides with the bound obtained in the proof of Theorem \ref{thm-twocl}, with $n$ replaced by $m=m_k$. As in that proof, there hence exist independent Rademacher variables $R_1,\ldots, R_m$ such that $Z_m=m^{-1}\sum_{1\le i<j\le m} R_iR_j$ satisfies
\[  \chi^2(\mathbb{Q}',\mathbb{P}') \le \mathbb{E}\exp\left[ 4m\te^2 |Z_m| \right]-1. \]
Provided $\te$ is defined as, for $a$ a small enough constant, 
\[ \te^2 =  a (4m)^{-1} \sim \Bigl(\frac{n}{k}\Bigr)^{-1},\]
 using Lemma \ref{lemcaos}  as in the proof of Theorem \ref{thm-twocl} leads to the bound 
$\| \mathbb{P}'-\mathbb{Q}' \|_1\le 1/2$ if $\theta^2$ is a small enough multiple of $k/n$, which again leads to a lower bound for the minimax risk of a positive constant times 
$k/n$, which proves \eqref{lbkfixed}. 
\end{proof}

\begin{proof}[Proof of Theorem \ref{thm-kcl}]
For $e=e_k$ and $M^\theta$ as in \eqref{propk}-\eqref{zte}, let
\begin{equation*}
\mathbb{Q}^\te = \frac1{k^n} \sum_{\vphi\in [k]^n} \bigotimes_{i<j} \text{Be}(M_{\vphi(i)\vphi(j)}^{\te}), 
\end{equation*}
and set $\mathbb{P}=\mathbb{Q}^{0}$ corresponding to $\te=0$. Our aim 
is to show that $\mathbb{Q}^\te$ and $\mathbb{P}$ are close in the sense $\|\mathbb{Q}^\te-\mathbb{P}\|_1\le 1/2$ say, while $\te$ is a fixed positive multiple of $\sqrt{k/n}$. For a given $\vphi\in [k]^n$, set \[ S_1:= 
\vphi^{-1}(\{1,2\})\qquad\text{ and }\qquad S_2:=\vphi^{-1}(\{3,4,\ldots,k\}). \]
By definition we have $S_1=S_2^c:=\{1,\ldots,n\}\setminus S_2$ and $|S_1|+|S_2|=n$.

The proof has two steps. First, one shows that with high probability one can restrict to designs (i.e. specific mapping $\vphi$'s) such that there 
are around $2n/k$ nodes that have label either $1$ or $2$. Second, we show that estimation with a random design  is `harder' than in the (easiest) 
typical  fixed design case. This argument is reminiscent of `information processing inequalities' encountered in information theory, although here 
a maximisation also takes place for not knowing the class labels. It is then important to maximise only over designs obtained from Step 1, in order for the lower bound rate to be $k/n$.         
                           
{\em Step 1.} One first shows that it is possible to restrict the sum in the definition of $\mathbb{Q}^\te$ and $\mathbb{P}$ to $\vphi$'s in the set
\begin{align*} \cA_n & =\left\{\vphi\in[k]^n, \quad \left|\,|\vphi^{-1}(\{1,2\})|- \frac{2n}{k} \right|\le \frac{n}{k} \right\} \\
& =\left\{\vphi\in[k]^n, \quad \left|\,|\vphi^{-1}(\{3,\ldots,k\})| - \frac{(k-2)n}{k} 
\right|\le \frac{n}{k} \right\}. 
\end{align*} 
The reason is that the large majority of sets $S_1$ have a cardinality of 
the order close to $n/k$. The proportion of $\vphi$'s not in $\cA_n$ among all possible $\vphi$'s is given by the probability of a binomial $Y\sim\text{Bin}(n,2/k)$ variable being farther than $n/k$ from its mean. By Bernstein's inequality, as $v:=\text{Var}[Y]=n(2/k)(1-2/k)$, for any $t>0$,
\[ \mathbb{P}\left[ \left|Y - \frac{2n}{k}\right| > t \right]  \le 2\exp\left\{
- \frac{t^2}{2v+t} \right\}.\]
Taking $t=n/k$ and setting  $R_n:=|\cA_n|$, we have just shown that
\[ 0\le 1-\frac{R_n}{k^n} \le 2e^{-\frac{n}{k}(5-\frac{8}{k})^{-1}}. \]
Now set
\[ \widetilde{\mathbb{Q}}^\te = \frac1{R_n} \sum_{\vphi\in\cA_{n}} \bigotimes_{i<j} \text{Be}(M_{\vphi(i)\vphi(j)}^{\te})\quad 
\text{and}\quad\widetilde{\mathbb{P}}= \widetilde{\mathbb{Q}}^0.\]
By the triangle inequality,
\[ \| \mathbb{Q}^\te-\mathbb{P}  \|_1 \le \| \mathbb{Q}^\te - \widetilde{\mathbb{Q}}^\te \|_1
+ \| \widetilde{\mathbb{Q}}^\te - \widetilde{\mathbb{P}}\|_1 + 
\| \widetilde{\mathbb{P}} -  \mathbb{P}\|_1.
\]
By Lemma \ref{approxmix},  $\| \mathbb{Q}^\te - \widetilde{\mathbb{Q}}^\te \|_1+
\| \widetilde{\mathbb{P}} -  \mathbb{P}\|_1$ is bounded above by $4(1-R_n/k^n)\le 
8e^{-\frac{n}{k}(5-\frac{8}{k})^{-1}}$.

{\em Step 2.} We now focus on bounding the middle term $\| \widetilde{\mathbb{Q}}^\te - \widetilde{\mathbb{P}}\|_1 $. Let $\Sigma_n$ denote the collection of subsets of $\{1,2,\ldots,n\}$ with $|S_2-(k-2)n/k|\le n/k$. For a given $S\in\Sigma$, let $\vphi_S=\vphi_{|S}$ denote the restriction of $\vphi$ to $S$. Below we use the notation $\sum_{\vphi_S}$ with the meaning that each term of the sum corresponds to a possible mapping $\vphi_S$, that is a given collection of values $(\vphi(i))_{i\in S}\in \{1,\ldots,k\}^{S}$. 
 
 To do so, we rewrite $\widetilde{\mathbb{Q}}^\te$ and $\widetilde{\mathbb{P}}$ as `mixtures of mixtures', by splitting the sum over $\vphi$ into a sum over $S_2, \vphi_{S_2}$ and $\vphi_{S_1}$ given $S_2$. Specifying $\vphi$ is equivalent to 
giving oneself $S_2$ (then $S_1=S_2^c$), $\vphi_{S_2}$ and $\vphi_{S_2^c}=\vphi_{S_1}$.
Denote 
\[ P_\vphi^\te(i,j) = \text{Be}(M_{\vphi(i)\vphi(j)}^{\te})\qquad\text{ 
and }\qquad
P_\vphi^\te = \bigotimes_{i<j} P_\vphi^\te(i,j).\]
For given $S_2$ and $\vphi_{S_2}$, set
 \[ T_{\vphi,S_2}^\te = \frac{1}{2^{n-|S_2|}} 
\sum_{\vphi_{S_1}\given S_2, \vphi_{S_2}} P_\vphi^\te, \] 
where one sums over all possible mappings $\vphi_{S_1}$, while $S_2$ and $\vphi_{S_2}$ are fixed.
We have
\[ \widetilde{\mathbb{Q}}^\te  = 
\sum_{S_2\in \Sigma_n,\, \vphi_{S_2}} \frac{2^{n-|S_2|}}{R_n} T_{\vphi,S_2}^\te.
\] 
Note that the above measures are normalised to be probability measures. Indeed, given $S_2\in\Sigma_n$, there are $2^{|S_1|}=2^{n-|S_2|}$ possible choices for $\vphi_{S_1}$. As $\widetilde{\mathbb{Q}}^\te$ is of  total mass one, we have
\[ 
\sum_{S_2\in \Sigma_n,\,  \vphi_{S_2}} \la_{S_2}=1
\qquad\text{ for }\quad
\la_{S_2}:=\frac{2^{n-|S_2|}}{R_n}.\]
Using the triangle inequality, one can bound
\begin{align*}
 \| \widetilde{\mathbb{Q}}^\te - \widetilde{\mathbb{P}}\|_1
& = \left\| \sum_{S_2\in \Sigma_n,\,  \vphi_{S_2}}  \la_{S_2} T_{\vphi,S_2}^\te 
-   \sum_{S_2\in \Sigma_n, \vphi_{S_2}} \la_{S_2} T_{\vphi,S_2}^0  \right\|_1 \\
& \le \sum_{S_2\in \Sigma_n, \, \vphi_{S_2}} \la_{S_2}
 \left\| T_{\vphi,S_2}^\te - T_{\vphi,S_2}^0\right\|_1 
  \le \max_{S_2\in \Sigma_n,\,  \vphi_{S_2}} \left \| T_{\vphi,S_2}^\te - 
 T_{\vphi,S_2}^0 \right\|_1.
\end{align*}
It is now sufficient to bound uniformly the above $L^1$-distance.
For simplicity, we denote
\[ T_{\vphi,S_2}^\te =\frac{1}{2^{n-|S_2|}} 
\sum_{\vphi_{S_1}\given S_2} P_\vphi^\te := \frac{1}{2^{n-|S_2|}} 
\sum_{\vphi_1} P_{\vphi_1,\vphi_2}^\te, \]
where $\vphi_2=\vphi_{S_2}$ and $\vphi_1=\vphi_{S_1}$, and $\vphi$ is 
the pair $(\vphi_1,\vphi_2)$. Set
 \[ \la_{S_1}=2^{-(n-|S_2|)}=2^{-|S_1|}.\]
Using the definition of $T_{\vphi,S_2}^\te$ above,
\begin{align*}
&\left\|T_{\vphi,S_2}^\te  - T_{\vphi,S_2}^0\right\|_1
 =   \left\|
\sum_{\vphi_1} \la_{S_1} P_{\vphi_1,\vphi_2}^\te - 
\sum_{\vphi_1} \la_{S_1} P_{\vphi_1,\vphi_2}^0 \right\|_1\\
&  = \left\|
\sum_{\vphi_1} \la_{S_1} \{ P_{\vphi_1,\vphi_2}^\te - 
\sum_{\vphi_1'} \la_{S_1} P_{\vphi_1',\vphi_2}^0 \}\right\|_1\\
 & \le \sum_{\vphi_1} \la_{S_1} 
 \left \| \sum_{\vphi'_1} \la_{S_1} P_{\vphi'_1,\vphi_2}^\te - P_{\vphi_1,\vphi_2}^0\right\|_1 \le 
 \max_{\vphi_1} \left\| \sum_{\vphi'_1} \la_{S_1} P_{\vphi'_1,\vphi_2}^\te -  P_{\vphi_1,\vphi_2}^0\right\|_1.
\end{align*}
Combining this with the previous bounds one deduces that 
\begin{align*}
\| \mathbb{Q}^\te-\mathbb{P}  \|_1 & \le 
 \| \widetilde{\mathbb{Q}}^\te - \widetilde{\mathbb{P}}\|_1 + 8e^{-\frac{n}{k}(5-\frac{8}{k})^{-1}}\\
& \le \max_{S_2\in \Sigma_n, \vphi(S_2)} \max_{\vphi_1} 
\ \left\|  
\sum_{\vphi'_1} \la_{S_1} P_{\vphi'_1,\vphi_2}^\te -  P_{\vphi_1,\vphi_2}^0\right\|_1
+ 8e^{-\frac{n}{k}(5-\frac{8}{k})^{-1}}.
\end{align*}
To conclude the proof, observe that the structure of the bound in the maximum in the last display is nearly identical to the quantities appearing in Equation \eqref{fixed-fbound} for the fixed-design case. 

In the present case, we have a fixed mapping $\vphi:\{1,\ldots,n\}\to\{1,\ldots,k\}$, with $\vphi_1=\vphi_{\given S_1}$ and $\vphi_2=\vphi_{\given S_2}$, that plays the role of $\vphi_0$ in the fixed-design case.  On the other hand, we have a collection of other mappings, say $\bar\vphi$, that coincide with $\vphi$ on $S_2$, that is $\bar\vphi_2=\bar\vphi_{\given S_2}=\vphi_{\given S_2}=\vphi_2$,
and that cover all possible cases for the image of $S_1$, namely $\bar\vphi_1=\bar\vphi_{\given S_1}=\vphi_1'$. 
The only difference to the fixed-design case is that $|S_1|$ belongs to $[n/k,3n/k]$, instead of being exactly $2\lfloor n/k \rfloor$, as specified in the definition of $\Sigma_n$ above. That is, denoting as above $Z_m=m^{-1}\sum_{1\le i<j\le m} R_iR_j$, with $m_1=|S_1|$,
\[  \left\|  
\sum_{\vphi'_1} \la_{S_1} P_{\vphi'_1,\vphi_2}^\te - 
P_{\vphi_1,\vphi_2}^0\right\|_1  \le \mathbb{E}\exp\left[ 4m_1\te^2 |Z_{m_1}| \right] - 1. 
\]
This bound is uniform over $S_2,\vphi_2$.  As $m_1\le 3n/k$, if one chooses $\te^2\le 1/(12\la n/k)$, with $\la=\la(1+\veps)$ the constant in Lemma \ref{lemcaosc}, then this Lemma implies that for any $m_1$ between $n/k$ and $3n/k$, the $L^1$-distance in the last display is bounded by $\veps$. 
 Crucially, the constant $\la$ in Lemma \ref{lemcaos} is independent of the number of terms in the Rademacher chaos. Deduce
 \[ \| \mathbb{Q}^\te-\mathbb{P}  \|_1 \le \veps +  8e^{-\frac{n}{k}(5-\frac{8}{k})^{-1}}. \]
Choosing $n/k>12$ makes this bound smaller than $\veps+4/5<1$ for $\veps<1/5$.  
 \end{proof}

\subsection{Useful lemmas}
\label{sec:lemmas}

Let $\{Z_i,\ i\ge 1\}$ be i.i.d. Rademacher variables.
For reals $x_{ij}$ and $N\ge 2$, set
\begin{align*}
 Y := Y_N & = \sum_{i<j\le N} y_{ij} Z_iZ_j,  \\ 
 s(Y)^2 & =\sum_{i<j\le N} y_{ij}^2. 
\end{align*}

\begin{lem}[Corollary 3.2.6 of \citet{dlpgine99}]  \label{lemcaos}
Let $N\ge 2$, and $Y=Y_N$ and $s(Y)$ as above.  For every $c>1$, there exists $\la=\la(c)>0$ independent of $N$ such that
\[ \mathbb{E}\exp\left[ \frac{|Y|}{\la s(Y)} \right] \le c. \]
\end{lem}

We repeatedly use Lemma \ref{lemcaos} in the case where all $y_{ij}$ are equal to $1$, for various values of $N$. In such a setting, a reformulation is as follows. For any $c>1$ and $N\ge 2$, one can find a constant $a=a(c)$ independent of $N$ such that 
\begin{equation} \label{caoshand}
 \mathbb{E}\exp\left[ a \frac{|Y_N|}{N} \right] \le c. 
\end{equation}

\begin{lem}[Rademacher  chaos with explicit constant]  \label{lemcaosc}
Let $N\ge 2$, and $Y=Y_N$ and $s(Y)$ as above.  For any $0\le\delta\le 1$,
\[ \mathbb{E}\exp\left[ \delta \frac{|Y|}{s(Y)} \right] -1 \le r(\delta), 
\quad\text{with}
\quad r(\delta)=\delta + \frac{\delta^2}{2} + \frac{8\delta^3}{6} + 
\frac{1}{1-e\delta}\frac{(e\delta)^4}{\sqrt{8\pi}}. \]
\end{lem}
The lemma applied with $\delta=1/3$ gives a bound $1.87$ for the right hand side.
\begin{proof}
Theorem 3.2.2 in \citep{dlpgine99} gives, for any $k\ge 2$,
\[ E |Y|^k \le (k-1)^k s(Y)^k. \]
For $k=1$ one has $E[|Y|] \le E[Y^2]^{1/2}=s(Y)$. From this one deduces that for any $0\le\delta\le 1$,
\begin{align*}
 E \exp \left[ \delta \frac{|Y|}{s(Y)} \right] 
& \le 1+ \delta + \frac{\delta^2}{2} + \frac{8\delta^3}{6} + 
\sum_{k\ge 4} \frac{(k-1)^k}{k!} \delta^k,
\end{align*}
and the result follows from an application of the nonasymptotic Stirling bound
 $k! \ge e^{-k} k^{k+\frac12}\sqrt{2\pi}$ valid for $k\ge 1$.
\end{proof}

\begin{lem}[Le Cam's method `point versus mixture'] \label{lem-lc}
Let $\cP=\{P_{M},\ M\in\cM\}$ be a collection of probability measures indexed by an arbitrary set $\cM=\{M_0,M_1,\ldots, M_N\}$, $N>1$. Set
\[ \mathbb P = P_{M_0},\quad \mathbb Q = \frac1N \sum_{k=1}^N P_{M_k}.\]
If $\psi$ is a real-valued functional such that $\psi(P_{M_0})=\tau$ and $\psi(P_{M_i})=\theta$ for any $i=1,\ldots,N$, then
\[ \inf_{\hat \theta} \sup_{M\in\cM} E_{P_M}(\hat\psi(X) - \psi(P_M))^2
\ge  \frac14(\theta-\ta)^2(1- \frac12 \| \mathbb P - \mathbb Q\|_1),\]
where the infimum is over all estimators $\hat\psi(X)$ of  $\psi(P_M)$ based on the observation of $X\sim P_M$.
\end{lem}
\begin{proof}
This is a standard variation on the case where $N=1$ stated in e.g. \citep{binyu}. 
\end{proof}

\begin{lem}[Bound on total variation distance] \label{lem-tv}
For $n\ge 1$, let $P_1,\ldots,P_n$ and $Q_1(k),\ldots,Q_n(k)$ for $1\le k\le N$, for some $N\ge 1$, be probability measures. Set
\[ P =\bigotimes_{i=1}^n P_{i},\quad Q(k) = \bigotimes_{i=1}^n Q_i(k)  \quad \mathbb{P} = P,
 \quad \mathbb{Q} = \frac1N \sum_{k=1}^N Q(k).\]
Suppose that for any $i$, $Q_i(k)$ has density $1+\Delta_i(k)$ with respect to 
$P_i$. Denote $\ta_i(k,l)=P_i\Delta_i(k)\Delta_i(l)$. Then, for $\chi^2(\mathbb{Q},\mathbb{P})=\int (d\mathbb{Q}/d\mathbb{P} - 1)^2 d\mathbb{P}$, 
\begin{align*}
\| \mathbb P - \mathbb Q\|_1^2 & \le \chi^2(\mathbb{Q},\mathbb{P})
 = \frac{1}{N^2} \sum_{k,l} \prod_{i=1}^n \{1+\ta_i(k,l)\} - 1.
\end{align*}
\end{lem}
\begin{proof}
The first bound on distances is standard, while the second bound follows from elementary calculations.
\end{proof}

\begin{lem} \label{approxmix}
Let $N,R$ be two integers with $N\ge 2$, $1\le R\le N$, and $(P_i)_{i\in I}$ be an arbitrary collection of probability measures with $|I|=N$. If 
$J\subset I$ and $|J|=R$, we have 
\[  \left\|\frac1N \sum_{i\in I} P_i - \frac1R \sum_{i\in J} P_i \right\|_1
\le 2\left(1-\frac{R}{N}\right). \]
\end{lem}
\begin{proof}
The result follows by splitting the sum over $I$ in a sum over $J$ and $I\setminus J$, applying the triangle inequality and using the fact that  $\|\sum_{j\in J}P_j\|_1=|J|$.
\end{proof}

\section{Proofs for results on graphon functionals} \label{sec:lbg}

To prove Theorem \ref{thm-quad}, we observe that polynomial graphons of bounded degree include the graphon $w_\theta$ in \eqref{polyw}.
The proof approximates this smooth graphon $w_\theta$ by a piecewise constant graphon, and then uses a lower-bound for such piecewise constants. We prove this lower bound, Lemma \ref{lem-kclasses}, first. Similar to the SBM case, this builds on Le Cam's point versus mixture method. We then proceed to prove Theorem \ref{thm-quad}.


\subsection{Auxiliary lower bound}

Assume the
function $w$ is piecewise constant, with different values taken along blocks corresponding to a regular partition of $[0,1]^2$ in $k\times k=k^2$ blocks, and $k$ an even integer $k=2l$, with $l\ge 1$. That defines a 
law of the form
\[ \frac{1}{k^n} 
\sum_{\vphi\in[k]^n} \bigotimes_{i<j} \text{Be}(Q_{\vphi(i)\vphi(j)})=P_{e_k,Q}, \]
where $\vphi$ is an element of $[k]^n$ and $Q=Q^\theta$ a given $k\times k$ matrix defined below. In the next statement and proof, $E_\te$ denotes the expectation under this distribution.
Denote by $O_k$ the $k\times k$ matrix with only ones as coefficients,
\[ 
O_k = 
\begin{bmatrix}
     1 & \cdots & 1 \\
     \vdots & \vdots & \vdots \\
     1 & \cdots & 1 
    \end{bmatrix},
\]
and, for a {\em symmetric} $l\times l$ matrix $A$ with coefficients $A_{ij}\in[0,1]$, define the  $k\times k=(2l)\times (2l)$ matrix 
\[ B := \left[
\begin{array}{c|c}
A  & -A \\ \hline
-A & A
\end{array}\right]. \]
We define $Q=Q^\theta$ as the $k\times k=(2l)\times (2l)$ matrix
\begin{equation} \label{defQ}
Q=Q^\theta= \frac12 \cdot O_k + \theta\cdot B = 
\frac12 \cdot O_k + \theta\left[
\begin{array}{c|c}
A  & -A \\ \hline
-A & A
\end{array}\right].
\end{equation}
\begin{lem} \label{lem-kclasses}
Let $k=2l$ be an even integer and $A$ an arbitrary symmetric $l\times l$ matrix. 
Let $Q=Q^\te$ be the matrix defined in \eqref{defQ}.
There exists a constant $c_3>0$ such that
\[ \inf_{\hat \theta} \sup_{\te\in(-1/2,1/2)} E_\te(\hat\theta(X) - \theta)^2
\ge \frac{c_3}{n},\]
where the infimum is over all estimators of $\te$ valid under $E_\te=E_{P_{e_k,Q^\te}}$.  
\end{lem}  

\begin{proof}[Proof of Lemma \ref{lem-kclasses}]
Let $\mathbb{Q} = P_{e_k,Q^\te}$ be as above. That is, 
\[  \mathbb Q = \frac1{k^n} \sum_{\vphi\in [k]^n}
P_{Z_\vphi^\theta}, \qquad P_{Z_\vphi^\te} = \bigotimes_{i<j}  \text{Be}(Q^{\theta}_{\vphi(i)\vphi(j)}), \]
with $\{Z_\vphi^\te\}$ the matrix of general term 
$z_{ij}(\vphi,\theta) = Q^{\theta}_{\vphi(i)\vphi(j)}$,  for $\vphi$ ranging over the set $[k]^n$.
Let $\mathbb{P}$ denote the Erd\"os-Renyi $ER(1/2)$ distribution over $n$ 
nodes, which also corresponds to $P_{e_k,Q^\te}$ for $\te=0$.
Consider the functional $\psi$ defined as, 
\[ \psi(P_{e_k,Q^\te})=\te. \]
By definition, for any $\vphi\in [k]^n$, we have $\psi(\mathbb{P})=\psi(P_{e_k,Q^0})=0$ and 
$\psi(\mathbb{Q})=\psi(P_{e_k,Q^\te})=\theta$. 
The same computation as in the proof of Theorem \ref{thm-twocl}  now shows that, for $B$ given in the  display below,
\[ \|\mathbb{P}-\mathbb{Q}\|_1^2 \le \frac{1}{k^{2n}} \sum_{\vphi,\psi \in [k]^n}  
\exp\Big\{ 4\theta^2 \sum_{i<j} B_{\vphi(i)\vphi(j)} B_{\psi(i)\psi(j)} \Big\} - 1. \]
The last term in the bound can be interpreted as an expectation over $\vphi, \psi$, where both variables are sampled uniformly from the set of all 
mappings from $\{1,\ldots,n\}$ to $\{1,\ldots,k\}$. Recall that $l=k/2$ 
and for any integer $s$, denote by $[s]_l$ the integer in  $\{1,\ldots,l\}$ that equals $s$ modulo $l$, plus $1$. 
The variable $B_{\vphi(i)\vphi(j)}$ can be written
\[ B_{\vphi(i)\vphi(j)} = (-1)^{\vphi(i)>l} (-1)^{\vphi(j)>l} 
A_{[\vphi(i)]_l[\vphi(j)]_l}. \]
When $\vphi$ follows the uniform distribution over $[k]^n$, the variables 
$(\vphi(i))_i$ are independent and are marginally  uniform over $[k]$. Also, the variables $((-1)^{\vphi(i)>l})_i$ and 
$([\vphi(i)]_l)_i$ are independent under the uniform distribution for $\vphi$ as we show next. If $\text{Pr}$ denotes the corresponding distribution, then
for any ${i\leq n}$ and any $s\leq l$ 
\[ \text{Pr}\Big[ \vphi(i)>l, [\vphi(i)]_l=s\Big] = \frac1k=\frac12\cdot \frac1l =\text{Pr}[ \vphi(i)>l]  \text{Pr}[\vphi(i)]_l=s].  \]
Note the identity holds both for $k\le n$ and $k>n$. 
Set   $R_i:=(-1)^{\vphi(i)>l}$ and $a_{ij}:=A_{[\vphi(i)]_l[\vphi(j)]_l}$. Deduce from the previous reasoning that the variables $(R_i)_i$  and $(a_{ij})_{i<j}$ are independent. Now, denoting by $E$ the expectation under $\text{Pr}$,
\begin{align*}
\|\mathbb{P}-\mathbb{Q}\|_1^2 & \le E\Bigg[  
\exp\Big\{ 4\theta^2 \sum_{i<j} B_{\vphi(i)\vphi(j)} B_{\psi(i)\psi(j)} \Big\}\Bigg] - 1\\
& \le E\Bigg[ E\Bigg[ 
\exp\Big\{ 4\theta^2 \sum_{i<j} a_{ij} R_iR_j  \Big\}\Bigg\vert a_{ij}\Bigg]\Bigg] - 1.
\end{align*}
As $(R_i)_i$  and $(a_{ij})_{i<j}$ are independent, one can compute the inner expectation in the last display under the distribution of $(R_i)_i$, 
the $a_{ij}$'s being fixed. The $(R_i)_i$ form a sample of independent Rademacher variables, hence 
\[Z_n:=\sum_{i<j} a_{ij} R_iR_j \] is a Rademacher chaos of order $2$ with weights $(a_{ij})$. 
Suppose the matrix $A$ is not identically zero (otherwise the bound below 
holds trivially). By Lemma \ref{lemcaos}, for any $c>1$ one can find $\la>0$ with
\[ E\exp\Bigl[ \frac{|Z_n|}{\la \|Z_n\|_2} \,\Big\vert\, a_{ij} \Bigr] \le c \qquad\text{ where }\qquad
\|Z_n\|^2_2=E[|Z_n|^2\given a_{ij}]=\sum_{i<j} a_{ij}^2.\] 
Choose $c=3/2$. By definition, all $a_{ij}$s are bounded by $1$. There is hence a $\la>0$ such that, if $\te^2=2/(\la n)$,
\[ E\exp\bigl[ 4\te^2 |Z_n|\ \,\big\vert\, a_{ij} \bigr]-1 \le 3/2-1=1/2. \] 
The result now follows  from an application of Lemma \ref{lem-lc} to the functional $\psi$.
\end{proof}

\subsection{Proof of the theorem}

\begin{proof}[Proof of Theorem \ref{thm-quad}]
Let us recall the definition, for any $0\le \theta\le 1$, of the function 
$w=w_\te$ in \eqref{polyw}
\[ w_\theta(x,y)=\frac12-\theta(x-\frac12)(y-\frac12)\] 
and let $\psg w_\te \psd$ be its graphon equivalence class. By definition, $\psg w_\te \psd$ belongs to $\cP$. One has
\[ \int_{[0,1]^2} w_\theta(x,y)^2 dx dy = \frac14 + c^2\theta^2, \]
for some constant $c>0$, so that $\ta(\psg w_\theta \psd)-\ta(\psg w_0 \psd)=c\theta$.
 The function $w_0$ is the constant $1/2$, and 
the density of the data distribution $P_{\psg w \psd}$  with respect to counting measure on $\{0,1\}^{n(n-1)/2}$ is
\[ p_w(\{x_{ij}\}_{i<j}) = \int \cdots \int  \prod_{i<j}\text{Be}(w(u_i,u_j))(x_{ij}) du_1\cdots du_n,  \]
where, for any $z$ in $[0,1]$ and $x_{ij}$ in $\{0,1\}$, we have set
\[ \text{Be}(z)(x_{ij}) = z^{x_{ij}} (1-z)^{1-x_{ij}}. \]
Next one 
shows that $P_{\psg w \psd}$ is close in the total variation sense to a discrete mixture of the previous Bernoulli-probability distributions, provided the number of points in the mixture is suitably large. 
To do so, we approximate the function $P_n$ defined by
\[ P_n: (u_1,\ldots,u_n) \mapsto \prod_{i<j}\text{Be}(w(u_i,u_j))(x_{ij}), \] 
by a piecewise constant function $h_{N,\theta}=h_N$, where $[0,1]^n$ is 
split into $N^n$ blocks, $N\ge 1$, using a regular grid of $[0,1]^n$ with 
points $(i_1/N,\ldots,i_n/N)$ and $0\le i_j \le N$ for all $j$. To do so, 
one just replaces $w(u_i,u_j)$ by, say, the value of $w$ on the 
middle of the 
 block the point $(u_i,u_j)$ belongs to. This defines a function 
\[ Q_{n,N}: (u_1,\ldots,u_n) \mapsto \prod_{i<j}\text{Be}(\bar w(u_i,u_j))(x_{ij}),  \]
where $\bar w$ is constant on every block of the subdivision.  
Let $Q_w^N$ denote the corresponding measure, with density
\[ q_w^N(\{x_{ij}\}_{i<j}) = \int \cdots \int  \prod_{i<j}\text{Be}(\bar w(u_i,u_j))(x_{ij}) du_1\cdots du_n.  \]
Taking $w=w_\theta$ as above,  the function $P_n$ is a polynomial in $u_1,\ldots,u_n$, and its degree with respect to each variable $u_i$ is $n-1$.  The partial derivatives of $P_n$ can be computed, and each of them can be seen to be bounded by ${n-1}$: For each variable, only $n-1$ non-zero terms appear when evaluating the partial derivative, and each term is uniformly bounded by $1$. Consequently, if $(u_1,\ldots,u_n)$ and $(u_1',\ldots,u_n')$ belong to the same block,
\[ |P_{n} (u_1,\ldots,u_n) - P_{n} (u_1',\ldots,u'_n)|\le 
(n-1)\sum_{i=1}^n |u_i-u_i'|\le n^2/N.\]
For $w=w_\te$ as above, we can thus bound the total variation distance as
\begin{align*}
 \| P_{\psg w \psd} - Q_w^N \|_1 & \le \sum_{{x\in\{0,1\}^{\frac{n(n-1)}{2}}} } |p_w-q_w^N|(x)\\
 & \le n^2 \max_{x\in\{0,1\}^{\frac{n(n-1)}{2}}} |p_w-q_w^N|(x)\le n^2 (n^2/N) = n^4/N.
\end{align*}
Each probability measure $Q_w^N$ is a mixture of $N^n$ distributions, each of which in turn corresponds to a block in the subdivision of $[0,1]^n$. One can rewrite
\[ Q_w^N = \frac{1}{N^n} \sum_{\vphi\in[N]^n}
\bigotimes_{i<j} \text{Be}(M_{\vphi(i)\vphi(j)}), \]
where the matrix $M=(M_{pq})_{1\le p,q \le N}$ is the symmetric matrix with terms
\[ M_{pq}=w_\theta\left(\frac{p-\frac12}{N},\frac{q-\frac12}{N}\right).\]  
If $N$ is even, which one can assume without loss of generality, the matrix $M$ is exactly of the same form as $Q$ in \eqref{defQ}, with elements in $(0,1)$, so one can use the bound in $\|\cdot\|_1$-distance between measures obtained in the proof of Lemma \ref{lem-kclasses}. Note that the argument remains valid even if the number of classes exceeds the number of 
observations $n$, which will be of importance below. For a small constant 
$c$ and $\theta^2=\kappa/n$, we obtain 
\[ \| Q_{w_\theta}^N - Q_{w_0}^N\|_1 \le c \]
for $\kappa$ sufficiently small. Choosing $N=Cn^4$, for $C>0$ large enough, leads to
\begin{align*}
\| P_{\psg w_\te \psd} - P_{\psg w_0 \psd} \|_1 & \le \| P_{\psg w_\te \psd} - Q_{w_\te}^N \|_1 + \|Q_{w_\te}^N - Q_{w_0}^N\|_1
+ \|Q_{w_0}^N-P_{\psg w_0 \psd}\|_1\\
& \le n^4/N + c +0 \le c'<1/2.
\end{align*}
An application of Lemma \ref{lem-lc} with the functional $\psi(P_{\psg w\psd}):=\psi(\psg w\psd)$ concludes the proof of the lower bound in Theorem \ref{thm-quad} in the case where $\psi(\cdot)=\ta(\cdot)$. The lower bound for a general $\psi$ follows by the same proof, noting that the specific form of the functional only comes in through the difference $\psi(\psg w_\te \psd)-\psi(\psg w_0\psd)$, which behaves as for $\ta(\cdot)$ by assumption.

For the upper-bound, we first
link the squared distance to the truth for the functional to the squared $L^2$-distance of corresponding graphons. Let $w,w_1$ be two fixed graphon functions, and suppose that at least one of these is non constant (almost everywhere), which means that either $\ta(\psg w\psd)>0$ or $\ta(\psg w_1\psd)>0$. Then, writing simply $\int$ to denote the double integral on 
$[0,1]^2$,
\begin{align*}
\ta(\psg w_1 \psd) - \ta(\psg w \psd) & = \left\{\int (w_1-\int w_1)^2\right\}^{1/2}
-  \left\{\int (w-\int w)^2\right\}^{1/2} \\
& =  \frac{\int  (w_1-\int w_1)^2 - \int (w-\int w)^2  }{ \left\{\int (w_1-\int w_1)^2\right\}^{1/2} +
 \left\{\int (w-\int w)^2\right\}^{1/2}},
 \end{align*}
where the denominator is nonzero by assumption on $w,w_1$; we henceforth denote it $c$. Then
\begin{align*}
\ta(&\psg w_1 \psd) - \ta(\psg w \psd)  \le c^{-1}\int \left\{ w_1-w - \int (w_1-w)\right\}\left\{ w_1+w - \int (w_1+w)\right\}\\
& \le c^{-1}\left[\int  \left\{w_1-w - \int (w_1-w)\right\}^2\right]^{1/2} 
\left[\int  \left\{w_1+w - \int (w_1+w)\right\}^2\right]^{1/2}.
\end{align*}
The two factors in brackets are bounded as follows: For the second term,
apply the inequality $(a+b)^2\le 2 a^2 + 2 b^2$, followed by $\sqrt{u+v}\le \sqrt{u}+\sqrt{v}$, which yields
\begin{align*}
&\left[\int  \left\{w_1+w - \int (w_1+w)\right\}^2\right]^{1/2}  \\
& \qquad \le 
\sqrt{2} \left[\int  (w_1-\int w_1)^2 +\int  (w - \int w)^2 \right]^{1/2}\le \sqrt{2} c.
\end{align*}
For the first term, use $0\le \int (g-\int g)^2\le \int g^2$ for a bounded measurable $g$, as one integrates over $[0,1]^2$. That yields
\[ \left\{\ta(\psg w_1 \psd) - \ta(\psg w \psd)\right\}^2  \le 2\int (w_1-w)^2,\] and this inequality clearly still holds true in case $\ta(\psg w 
\psd)=\ta(\psg w_1 \psd)=0$. One concludes that
\begin{align*}
  &\left\{\ta(\psg w_1 \psd) - \ta(\psg w \psd)\right\}^2  \\
 & \qquad  \le 2
 \inf_{T\in\cT} \int\int_{[0,1]^2} \left|w_1(T(x),T(y)) - w(x,y) \right|^2dxdy=:\delta^2(w_1,w), 
 \end{align*}
 where $\cT$ is the set of all measure-preserving bijections of $[0,1]$. Indeed, the previous inequalities hold true for any choice of representer 
of the graphon $w_1$, so one can take the infimum over $\cT$ in the previous bounds. By Corollary 3.6 of \cite{kvt16}, for data $X$ generated from 
$P_w$, there exists an estimator $\hat w=\hat w(X)$ that satisfies $E_{P_w}[\delta^2(\hat w,w)]\le C(\log{n}/n)$. Since $\psg w \psd$ has a representer that belongs to $\cP_B$ by assumption, it belongs in particular to the H\"older class $\Sigma(1,L)$, provided $L$ is chosen large enough. For the plug-in estimator $\hat\ta(X):=\ta(\hat w)$, combining the previous result with the last display implies 
 \[ E\left[\{\hat\ta(X)-\ta(\psg w \psd)\}^2\right]\le C\frac{\log{n}}{n},\]
 for $C$ large enough depending only on $\cP_B$, which concludes the proof.
\end{proof}

\section*{Acknowledgements}
I. C. is very grateful for the hospitality of Columbia's statistics department, where parts of this work where carried out. I. C.'s work is supported by ANR-17-CE40-0001 (BASICS).

\bibliographystyle{plainnat}
\bibliography{grbib}

\newpage

\appendix

\section{Upper bounds computable in polynomial-time}
\label{sec:spec:k}

This section generalizes the polynomial-time estimate in Section
\ref{sec:spec:2} to ${k\geq 2}$ classes, by combining the 
spectral clustering method of \citet{Lei:Rinaldo:2015:1} with a
refinement due to \citet{Lei:Zhu:2017:1}. The latter is based on a
sample splitting, and under appropriate conditions on the connectivity
matrix recovers the labels {\em exactly}, with high probability.
Theorem \ref{thmspec} below shows that, under additional conditions, 
this polynomial-time estimator achieves the minimax rate.

\subsection{Spectral estimation for ${k\geq 2}$ classes}
Recall the assumed form of the connectivity matrix $M^\te$ in \eqref{zte}. The conditions of the next results are in terms of an `aggregated' $(k-1)\times (k-1)$ matrix $N$ obtained from $M^\theta$ by merging the first and second row/columns when $\te=0$, that is
\[ N = \begin{bmatrix}
     1/2 &  a_1 & \cdots & a_{k-2} \\
     a_1  & b_{11} & \cdots & b_{1 k-2}  \\
  \vdots & \vdots & & \vdots \\
 a_{k-2} & b_{1 k-2} & \cdots & b_{k-2 k-2}
\end{bmatrix}.\]
Recall that $\vphi$ denotes the true labelling map.
Define a labelling $\psi:[n]\to[k-1]$ by $\psi(v)=1$ if $\vphi(v)\in\{1,2\}$ and
$\psi(v)=\vphi(v)-1$ if $v\in \{3,\ldots,k\}$. That is, we `aggregate'
nodes of label $1$ or $2$ in one class and renumber the remaining
labels so that the label set is, now, $[k-1]$. Following
\cite{Lei:Zhu:2017:1}, we write $g_v=\psi(v)$ for the true (aggregated) label of node $v\in[n]$ and $\cI^{(l)}=\{v\in[n]:\ g_v=l\}$.

The algorithm {\tt Spec-$\theta$} specified in the frame below has three steps. First, one runs the exact label recovery algorithm {\tt V-Clust} of \citet{Lei:Zhu:2017:1} for $K=k-1$ classes. 
Under some conditions on the matrix $N$, see (A1)--(A2) below, this finds 
the `aggregate' labels $\psi$ above {\em up to label permutation} with high probability. Then the aim is to recover the aggregated class with original labels $1$ and $2$. Due to the label switching issue, this requires some extra condition on $N$. For simplicity (see also comments below) we assume in (A3) that the diagonal terms $b_{ii}$ are separated from $1/2$, 
which enables to estimate the aggregated class label $1$ by comparing diagonal empirical connectivities to $1/2$. Finally, in a third step one  can run the spectral algorithm $\cS_2$ from Section \ref{sec2} on the nodes 
found at the previous step. 
\vp

\fbox{\begin{minipage}{0.9\textwidth}
   \begin{center}
{\tt   Algorithm: Spectral method for estimation of $\theta$ (Spec-$\theta$)}
   \end{center}
{\bf Input:} adjacency matrix $X$ (where we set $X_{ii}=0$), number of classes $k$\\

{\bf Subroutines:}  {\tt V-Clust} (Lei-Zhu), Initial community recovery $\cS$ (Lei-Rinaldo), Spectral algorithm $\cS_2$ for $k=2$ (Section \ref{sec2})
   \begin{enumerate}
      \item Apply {\tt V-Clust} on adjacency matrix $X$ using $k-1$ classes, $\cS$ and $V=2$
      \[ \hat g = \text{{\tt V-Clust}}(X,k-1,V,\cS). \] 
      \item Set 
       $\hat\cI^{(1)}=\{v\in[n]:\ \hat g_v=\hat \ell\,\}$, where 
      \[ \hat \ell \, = \,\underset{l\in[k-1]}{\text{argmin}}\
      \Bigg|\,
      \frac{1}{{{|\hat{g}^{-1}(l)|} \choose {2}}}\sum_{i<j,\, i,j\in \hat{g}^{-1}(l)} X_{ij}\,
       -\,\frac12
      \,\Bigg|
        \]
      \item Run spectral algorithm $\cS_2$ for $k=2$ on corresponding nodes and set
      \[ \hat\theta = \cS_2(X^{\hat \cI^{(1)}}),\]
      where $X^{\hat \cI^{(1)}}$ is the induced adjacency matrix over nodes in $\hat \cI^{(1)}$.
  \end{enumerate}
\end{minipage}}
\vp
 
We set $K=k-1$ and assume that, for a large enough universal constant $C$:
\begin{enumerate}
\item[(A1)] $N$ is full rank
and any two rows of $N$ are separated by at least $\ga=\ga(K)>0$ in $\ell_2$-norm.
\item[(A2)] For $\la=\la(K)$ the smallest absolute eigenvalue of $N$,
\begin{equation*}
n \la(K) \ga(K)\ge CK^{4.5},\quad n\ga(K)^2\ge CK^3\log{n},\quad
n\ge CK^3.
\end{equation*}
\item[(A3)] For all $i\in\{1,\ldots,k-2\}$,
\[ |b_{ii}-1/2|\ge \kappa, \]
where $\kappa=\kappa(K)\ge C\sqrt{K(\log{n})/n}$.
\end{enumerate}
Comments on (A1)--(A3) follow below. For a version for sparse graphs, see 
Appendix \ref{sec_spa}.

\begin{thm} \label{thmspec}
In the fixed design SBM model with $k$ classes, under the assumptions (A1)--(A3),
let us set, for $c$ a small enough universal constant and $K=k-1$,
\begin{equation} \label{def_tk}
T_K := c \frac{\la(K) \ga(K)^{1/2}}{K^{5/4}} \wedge \frac{\kappa}{4}.
\end{equation}
Then the obtained $\hat \theta$ from algorithm {\tt Spec-$\theta$} satisfies, for $C_3$ a  large enough constant,
\[ \sup_{|\te|\le T_K,\, \vphi\in\Sigma_e} E_{\te,\vphi}\left[ \hat\te-\te\right]^2 \le C_3\frac{k}{n}. \]
\end{thm}
\begin{proof}
  This is a special case of Theorem \ref{thm-kcl-spa}, in
  Appendix \ref{sec_spa} below.
\end{proof}

The algorithm {\tt Spec-$\theta$}, unlike the likelihood method considered below, only uses the fact that the connectivity matrix is of the form $M^\theta$, but does not use specific knowledge of the vector $a$ and matrix $B$ to compute $\theta$.\\[.5em]
{\noindent\bf Comments on the assumptions}.
Conditions (A1) and (A2) are typical for spectral methods; their specific 
form is that assumed by \citet{Lei:Zhu:2017:1}, with  the initial recovery algorithm being that of \citet{Lei:Rinaldo:2015:1}. If  $K$ is fixed independently of $n$, then (A2) follows from (A1) if $n$ is large enough. 
Condition (A3) is specific to our problem, and assumed in this form
only for simplicity of exposition: To identify the special cluster
arising from the $1/2$ coefficient in the matrix $N$ (step 2. in {\tt
  Spec-$\te$}), some identifiability condition is needed, because even
the refined spectral clustering algorithm of \cite{Lei:Zhu:2017:1} can
only recover the original labels up to a permutation. Condition (A3)
is similar in spirit to condition \eqref{techm}, but weaker.
It can be replaced with any 
other condition that ensures cluster $1$ can be identified from a  noisy, 
permuted version of $N$ (with noise amplitude going to zero fast, as $k/n$). Note that, if $k$ is fixed and $n$ large enough, (A3) simply requires 
the diagonal terms of $B$ to differ from $1/2$.

Finally, a comment on $T_K$ in \eqref{def_tk}.  The label recovery in Steps 1--2 is run with $k-1$ classes, and hence joins two of the $k$ classes 
in the sample. The restriction on the range of $\te$ ensures the classes joined are the first two, with high probability.
Indeed, here we are interested in the situation where $\te$ may be small, 
which makes identification of labels difficult, and the rate slow; if $\te$ is large, the problem becomes easier. Again, note that if $k$ is fixed, the condition simply requires that $|\te|$ is smaller than a given constant.

\subsection{Simulation study}
\label{sec:simulation}

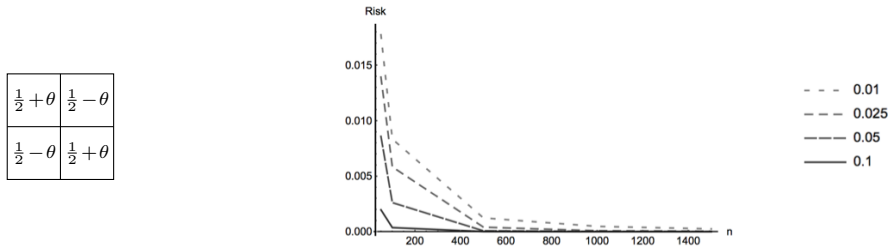
\begin{figure}
  \makebox[\textwidth][c]{
  \begin{tikzpicture}
    \begin{scope}[scale=3.5]
    \draw (0,0)--(.4,0)--(.4,-.4)--(0,-.4)--cycle;
    \draw (.2,0)--(.2,-.4);
    \draw (0,-.2)--(.4,-.2);
    \node at (.1,-.1) {\scriptsize $\frac{1}{2}\!+\!\theta$};
    \node at (.3,-.1) {\scriptsize $\frac{1}{2}\!-\!\theta$};
    \node at (.3,-.3) {\scriptsize $\frac{1}{2}\!+\!\theta$};
    \node at (.1,-.3) {\scriptsize $\frac{1}{2}\!-\!\theta$};
    \end{scope}
    \begin{scope}[xshift=4.2cm,yshift=-.7cm]
      \node at (0,0) {\includegraphics[height=3.2cm]{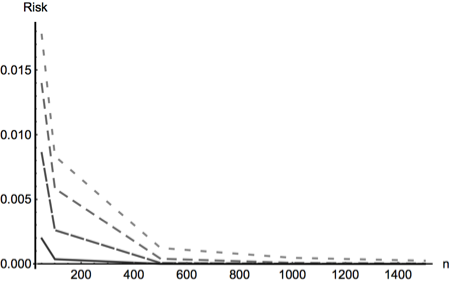}};
      \node at (5.3,0)
            {\includegraphics[height=3.2cm]{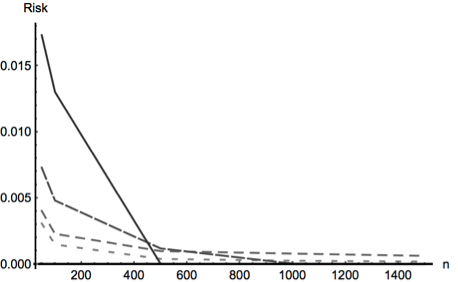}};
            \node at (7,1) {\includegraphics[height=1cm]{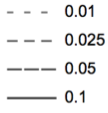}};
      \node at (0,-2) {\scriptsize (i) spectral estimator};
      \node at (5.3,-2) {\scriptsize (ii) sample splitting};
    \end{scope}    
  \end{tikzpicture}
  }
  \caption{Estimation of $\theta$ in the two-class case, using (i) the spectral estimator \eqref{te_spec} and 
    (ii) sample splitting. Graphs of size ${n=50,100,500,1000,1500}$ are generated from the graphon
    on the right, for ${\theta=0.01,0.025,0.05,0.1}$. Shown is the empirical risk (computed over
    1000 experiments) as a function of the sample size $n$.}
    \label{fig:k:2}
\end{figure}

Those estimators described above that are computationally feasible---the spectral and sample splitting estimators
for $k=2$, and the Spec-$\theta$ estimator for $k>2$---can be tested in 
simulation: Draw $n$ vertices from
a stochastic block model as in \eqref{modk} with a given value of $\theta$, compute the respective estimate, and report the 
empirical quadratic risk. Figure \ref{fig:k:2} shows how the risk develops as a function of sample size
for different values of $\theta$, for the two-community model \eqref{simat}.
For $k>2$ communities, the model is given by the connectivity matrix \eqref{zte}.
Simulation results for $k=5$, with ${a_1=\frac{1}{12}}$, ${a_2=\frac{11}{12}}$ and ${a_3=1}$, are shown in 
Figure \ref{fig:k:5}.  
As is visible in   Figures \ref{fig:k:2} and \ref{fig:k:5}, smaller values of $\theta$ correspond overall to a larger risk, and a much slower decay of the empirical risk curves. This illustrates our theoretical finding that there exists a range of parameters corresponding to two classes that become close where estimation is much slower. 

\begin{figure}
\makebox[\textwidth][c]{
  \begin{tikzpicture}
    \begin{scope}[scale=0.8, every node/.style={scale=0.8}]
    \begin{scope}[scale=3.5]
    \draw (0,0)--(1,0)--(1,-1)--(0,-1)--cycle;
    \draw (.2,0)--(.2,-1);
    \draw (.4,0)--(.4,-1);
    \draw (.6,0)--(.6,-1);
    \draw (.8,0)--(.8,-1);
    \draw (0,-.2)--(1,-.2);
    \draw (0,-.4)--(1,-.4);
    \draw (0,-.6)--(1,-.6);
    \draw (0,-.8)--(1,-.8);
    \node at (.1,-.1) {\scriptsize $\frac{1}{2}\!+\!\theta$};
    \node at (.3,-.1) {\scriptsize $\frac{1}{2}\!-\!\theta$};
    \node at (.3,-.3) {\scriptsize $\frac{1}{2}\!+\!\theta$};
    \node at (.1,-.3) {\scriptsize $\frac{1}{2}\!-\!\theta$};
    \node at (.1,-.5) {\scriptsize $\frac{1}{12}$};
    \node at (.3,-.5) {\scriptsize $\frac{1}{12}$};
    \node at (.5,-.5) {\scriptsize $\frac{1}{12}$};
    \node at (.5,-.3) {\scriptsize $\frac{1}{12}$};
    \node at (.5,-.1) {\scriptsize $\frac{1}{12}$};
    \node at (.7,-.1) {\scriptsize $\frac{11}{12}$};
    \node at (.7,-.3) {\scriptsize $\frac{11}{12}$};
    \node at (.7,-.5) {\scriptsize $\frac{11}{12}$};
    \node at (.7,-.7) {\scriptsize $\frac{11}{12}$};
    \node at (.5,-.7) {\scriptsize $\frac{11}{12}$};
    \node at (.3,-.7) {\scriptsize $\frac{11}{12}$};
    \node at (.1,-.7) {\scriptsize $\frac{11}{12}$};
    \node at (.9,-.1) {\scriptsize $1$};
    \node at (.9,-.3) {\scriptsize $1$};
    \node at (.9,-.5) {\scriptsize $1$};
    \node at (.9,-.7) {\scriptsize $1$};
    \node at (.9,-.9) {\scriptsize $1$};
    \node at (.7,-.9) {\scriptsize $1$};
    \node at (.5,-.9) {\scriptsize $1$};
    \node at (.3,-.9) {\scriptsize $1$};
    \node at (.1,-.9) {\scriptsize $1$};
    \end{scope}
    \end{scope}
    \begin{scope}[xshift=8.5cm,yshift=-1.4cm]
      \node at (0,0) {\includegraphics[height=3.2cm]{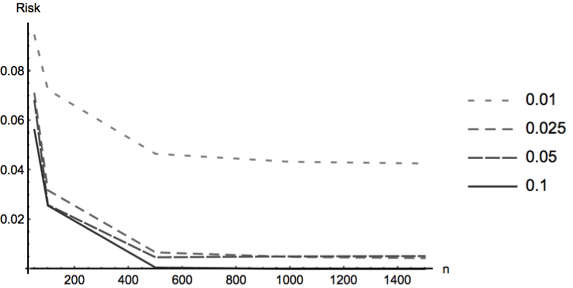}};
    \end{scope}
  \end{tikzpicture}
}
  \caption{Estimation of $\theta$ in the five-class case, using the Spec-$\theta$ algorithm. 
    Graphs of size ${n=50,100,500,1000,1500}$ are generated from the graphon
    on the right, for ${\theta=0.01,0.025,0.05,0.1}$. Shown is the empirical risk (computed over
    1000 experiments) as a function of the sample size $n$. For small values of $\theta$, convergence 
  slows visibly.}
  \label{fig:k:5}
\end{figure}

\section{Extension to sparse graphs} \label{sec_spa}

So far, we have for simplicity considered {\em dense} graphs, in the sense that at least some elements of the connectivity matrix (e.g. $1/2+\te$ or 
$1/2-\te$) are bounded away from zero.  

\subsection{Two classes}
An $\al_n$--sparse SBM model is generally defined as one in which the connectivity matrix $M$ can be written, for $\al_n$ a sequence going to $0$ with $n$, as $M = \al_n M_0$, 
for $M_0$ a nonnegative symmetric matrix with maximum entry $1$ \citep[e.g.][]{bickeletal11,Lei:Rinaldo:2015:1}. Here, we assume that the connectivity matrix is $M^\te(\al_n)$ with 
\begin{equation} \label{zte_spa}
M^\te(\al_n) = \al_n M^\te,
\end{equation}
and $M^\te$ as in \eqref{zte}. Then the largest coefficient of $M^{\te}$ is between $\al_n/2$ and $\al_n$, as the coefficients of the upper $2\times 2$ block are $\al_n(1/2\pm\te)$. We also set, for $\te\in[-1/2,1/2]$,
\begin{equation}
 Q^\te(\al_n)  = \al_n Q^\te, \qquad 
Q^\theta  \ = 
\begin{bmatrix}   
     \frac12 + \theta & \frac12 - \theta\\
     \frac12 - \theta  & \frac12 + \theta 
\end{bmatrix}. \label{simatspa}
\end{equation}
In constructing upper bounds below, we assume that for $C_s$ a large enough constant,
\[ (B0)\quad\qquad \al_n\ge C_s \frac{\log{n}}{n}, \]
as up to a constant $\log{n}/n$ is the typical boundary between the moderately sparse and very sparse situations, the later requiring different tools, see \cite{Lei:Rinaldo:2015:1}. For simplicity we also assume that $\al_n$ is known for the upper-bound results.
\begin{thm} \label{thm-twocl-spa}          
Consider a stochastic blockmodel \eqref{modsbm} with $k=2$ specified by 
$P_\te=P_{e,Q^\te(\al_n)}$ with $e,Q^\te(\al_n)$ given by  \eqref{prop2}-\eqref{simatspa}. There exists a constant $c_1>0$ such that for all $n\ge 2$,
\[ \inf_{T} \sup_{\te\in[-1/2,1/2]} E_{\te} \left[ T(X) - \theta \right]^2       
\ge c_1\left(1\wedge \frac{1}{n\al_n}\right), \]                          
where the infimum is taken over all estimators $T$ of $\te$ in the model $\cM$. 
Furthermore, if $\Delta_n=X-\al_nJ/2$, and $\la_1^a(\Delta_n)$ the largest absolute eigenvalue of $\Delta_n$, set $\tilde\te:=\la_1^a(\Delta_n)/\{(n-1)\al_n\}$. Then, under (B0), for some constant $C>0$ and $n\ge 2$, 
\[ \sup_{\te\in[-1/2,1/2]} E_\te[(\tilde\te-\te)^2] \le \frac{C}{n\al_n}. 
 \]
\end{thm}  

\subsection{$k\geq 2$ classes}

The case of $k$ classes carries over to the sparse situation as follows. The lower bound result is only modified by a scaling factor $1/\al_n$. For upper bounds, considering the more easily computable spectral algorithm 
{\tt Spec-$\te$} only,  Assumption (A2) is replaced by (B2) below, where $N$ has the same definition as in Appendix \ref{sec:spec:k}. 
\begin{enumerate}
\item[(B2)] For $\la=\la(K)$ the smallest absolute eigenvalue of $N$,
there exists  $C>0$ such that
\begin{equation*}
n \al_n\la(K) \ga(K)\ge CK^{4.5},\quad n\al_n\ga(K)^2\ge CK^3\log{n},\quad
n\ge CK^3.
\end{equation*}
\end{enumerate}

\begin{thm} \label{thm-kcl-spa}          
Consider a stochastic blockmodel \eqref{modsbm} with $k\ge 2$ classes specified by $\cM_k$ in \eqref{modk}, that is $P_\te=P_{e_k,M^\te}$ with $e_k,M^\te$ given by  \eqref{propk}--\eqref{zte_spa}, for fixed matrices $A, B$ with arbitrary coefficients. 
There exists a  constant $c_3=c_3(\rho)>0$, independent of $A,B$,  such 
that,  for all $n\ge 12k$,
\[ \inf_{T} \sup_{\te\in[-1/2,1/2]} E_{\te} \left[ T(X) - \theta \right]^2       
\ge c_3\left(1 \wedge \frac{k}{n\al_n}\right), \]                         
where the infimum is taken over all estimators $T$ of $\te$ in the model $\cM_k$. 
Let $\cS_{2,\al_n}$ be the algorithm for $k=2$ classes in the sparse case described in Theorem \ref{thm-twocl-spa}. Consider the fixed-design setting and suppose $(B0), (B2), (A1)$ and $(A3)$ are satisfied. Then  the algorithm {\tt Spec}-$\te$ used with subroutine $\cS_{2,\al_n}$ outputs an estimator $\hat\te$ that satisfies, for $T_K$ as in \eqref{def_tk}, 
\[ \sup_{|\te|\le T_K,\, \vphi\in \Sigma_e} E_{\te,\vphi}\left(\hat\te-\te\right)^2 
\le C \frac{k}{n\al_n}.\]
\end{thm} 

Similar comments as for Theorems \ref{thm-kcl}--\ref{thmspec} can be made. Also, in the case that $k$ does not grow with $n$, then (B2) follows from (B0) for $n$ larger than a fixed constant. The proof of the lower bound in Theorem \ref{thm-kcl-spa} is similar to that of Theorem \ref{thm-kcl} using the normalisation as in the proof of Theorem \ref{thm-twocl-spa} and is omitted.  The upper bound result includes that of Theorem \ref{thmspec} and is proved in Appendix \ref{sec:ub_spec}.

\section{Remaining proofs: likelihood-based upper bounds} \label{sec-ub}

The proof for ${k=2}$ below analyzes the least-squares criterion
directly. For ${k\geq 2}$ classes, we `isolate' the part corresponding
to the $2\times 2$ submodel, by controlling the number of errors in
recovering the labels of the corresponding $2$ classes. We then invoke the result for the case $k=2$.

\subsection{Interpretation as a pseudo-likelihood}

We first justify the interpretation of the estimator $\hat\theta$ in \eqref{pseml} as a maximum (pseudo-)likelihood estimate. In the fixed design model, suppose the data is Gaussian $\cN(\te_{ij},1)$ instead of Bernoulli Be$(\te_{ij})$. This suggests defining a (pseudo-)log-likelihood $\ell_n(\sigma,\te)$ as follows, with $c_n={n\choose 2}\log(2\pi)$,
\begin{align*}
-&2\ell_n(\sigma,\te)  = \sum_{i<j} (X_{ij} - Q_{\sigma(i)\sigma(j)}^\te)^2 + c_n\\
& = \sum_{i<j,\ \sigma(i)=\sigma(j)} (X_{ij} - (\frac12+\te))^2 +
\sum_{i<j,\ \sigma(i)\neq \sigma(j)} (X_{ij} - (\frac12-\te))^2 + c_n\\
& =  {n \choose 2} \te^2 +\te\left(
\sum_{i<j,\ \sigma(i)=\sigma(j)} (1-2X_{ij}) -
\sum_{i<j,\ \sigma(i)\neq \sigma(j)} (1-2X_{ij}) 
\right)+ C_n(X),
\end{align*}
for a constant $C_n(X)$ depending only on $n$ and $X$. Setting $b_n={n\choose 2}$ and 
\[ 2Z_n(\sigma,X) := - \sum_{i<j,\ \sigma(i)=\sigma(j)} (1-2X_{ij}) +
\sum_{i<j,\ \sigma(i)\neq \sigma(j)} (1-2X_{ij}), \]
it is enough to study 
the function $g_n(\te,\sigma):= b_n \te^2 - 2Z_n(\sigma,X) \te$, which satisfies
\[ g_n(\te,\sigma)=b_n\left(\te-\frac{Z_n(\sigma,X)}{b_n}\right)^2 - \frac{Z_n(\sigma,X)^2}{b_n} \ge - \frac{Z_n(\sigma,X)^2}{b_n}.\]
Consequently, the pseudo maximum likelihood estimator $(\hat\te,\hat\sigma)$ is given by \eqref{pseml} as claimed.

\subsection{Upper bound result, two classes}
\label{sec-ub:2}

\begin{proof}[Proof of Theorem \ref{thmub}]   
We first prove the result in the fixed design case. Let $\te_0,\sigma_0$ denote the true values of $\te,\vphi$. The aim is to show that $E_{\te_0,\sigma_0} (\hat\te-\te_0)^2\le C/n$ holds uniformly in $\te_0,\sigma_0$.
For a given $\sigma\in 2^{[n]}$,
\begin{equation*}     
 r_{ij} :=X_{ij}-1/2+(-1)^{\1_{\sigma_0(i)= \sigma_0(j)}}\te_0 \quad\text{and}\quad
 R_n(\sigma) := \sum_{\sigma(i)=\sigma(j)}r_{ij} - \sum_{\sigma(i)\neq\sigma(j)}r_{ij}. 
\end{equation*}      
One can write, for any $\sigma\in 2^{[n]}$,
 \begin{align*}
  Z_n(\sigma_0,X) & = b_n\te_0
+\sum_{i<j} (-1)^{\1_{\sigma_0(i)\neq \sigma_0(j)}} r_{ij} = b_n\te_0+R_n(\sigma_0),\\
Z_n(\sigma,X) &  = \te_0 \delta(\sigma,\sigma_0) + R_n(\sigma),
\end{align*} 
where we have set  
 \begin{equation}\label{deltas}
   \delta(\sigma,\sigma_0) = \sum_{i<j} (-1)^{\1_{\sigma_0(i)\neq \sigma_0(j)}}
   (-1)^{\1_{\sigma(i)\neq \sigma(j)}}.
 \end{equation}  
For any $t>0$ and $t_n=M_2/\rn$, and for a large enough $M_2$ to be chosen below,
\begin{align*}
& \lefteqn{P_{\te_0}[\rn|\hat\te-\te_0|\ge t]} &\\
& \qquad =  P_{\te_0}[\rn|\hat\te-\te_0|\ge t] \1_{|\te_0|\le t_n}+ 
 P_{\te_0}[\rn|\hat\te-\te_0|\ge t] \1_{|\te_0|> t_n} \\
& \qquad =:  \qquad\qquad \cP_1(t) \qquad\qquad+\qquad\qquad \cP_2(t). 
\end{align*}
By definition of $\hat\te$, with $\delta(\sigma,\sigma_0)$ defined in \eqref{deltas}, 
\[
b_n\hat\te = Z_n(\hat\sigma,X) = \te_0\delta(\hat\sigma,\sigma_0) 
+ R_n(\hat\sigma)
\qquad\text{ and }\qquad
 Z_n(\sigma_0,X) = \te_0 b_n + R_n(\sigma_0).
\]
For any $t\ge 4M_2$, using that $|\delta(\hat\sigma,\sigma_0)|\le b_n$,
\begin{align*}
\cP_1(t) \quad & \le \quad P_{\te_0}\left[ \rn|\te_0|
\left|\frac{\delta(\hat\sigma,\sigma_0)}{b_n}-1\right|
+\frac{\rn}{b_n}|R_n(\hat\sigma)| \ge t \right]\1_{|\te_0|\le t_n}\\
 & \le  \quad P_{\te_0}\left[\frac{\rn}{b_n}|R_n(\hat\sigma)| \ge t -2M_2 
\right]\\
 & \le \quad  P_{\te_0}\left[\sup_{\sigma\in2^{[n]}} |R_n(\sigma)| \ge \frac{b_n}{\rn}\frac{t}2 \right]. 
\end{align*}
For $\cP_2(t)$, there are two cases, depending on the sign of $\te_0$,
\begin{align*}
&\cP_2(t) \quad  \le \quad P_{\te_0}\left[ \rn|\hat\te-\te_0| \ge 
t \right] \1_{\te_0> t_n}
+ P_{\te_0}\left[ \rn|\hat\te-\te_0| \ge t \right]\1_{\te_0< - t_n} &\\
 & \le   \Big\{P_{\te_0}\left[\rn|\hat\te-\te_0| \ge t ,\, Z_n(\hat\sigma,X) \ge 0\right]
 \\
 &\qquad + P_{\te_0}\left[\rn|\hat\te-\te_0| \ge t ,\, Z_n(\hat\sigma,X) < 0\right]\Big\}
  \1_{\te_0> t_n}\\
 &   \qquad +  \Big\{P_{\te_0}\left[\rn|\hat\te-\te_0| \ge t ,\, Z_n(\hat\sigma,X) \ge 0\right] \\
& \qquad + P_{\te_0}\left[\rn|\hat\te-\te_0| \ge t ,\, Z_n(\hat\sigma,X) < 0\right]\Big\}
  \1_{\te_0< -t_n}
\end{align*}
Let us discuss the term $\te_0>t_n$ first and note that if $Z_n(\hat\sigma,X)\ge 0$, then $Z_n(\hat\sigma,X)=|Z_n(\hat\sigma,X)|\ge |Z_n(\sigma_0,X)|\ge Z_n(\sigma_0,X)$ using the definition of $\hat \sigma$ as a maximum.  First,
\begin{align*}
\lefteqn{ P_{\te_0}\left[\rn|\hat\te-\te_0| \ge t , 
Z_n(\hat\sigma,X) \ge 0\right]\1_{\te_0> t_n} }& \\
& \le P_{\te_0}\left[\frac{\rn}{b_n}| Z_n(\hat\sigma,X)- Z_n(\sigma_0,X) +
 R_n(\sigma_0)| \ge t , \, 
Z_n(\hat\sigma,X) \ge 0\right]\1_{\te_0> t_n} \\
& \le P_{\te_0}\left[\frac{\rn}{b_n}( Z_n(\hat\sigma,X)- Z_n(\sigma_0,X) ) +
 \frac{\rn}{b_n}|R_n(\sigma_0)| \ge t \right]\1_{\te_0> t_n} \\
 & \le P_{\te_0}\left[\frac{\rn}{b_n} \te_0 (\delta(\hat\sigma,\sigma_0) - b_n) +
 \frac{\rn}{b_n} R_n(\hat\sigma) + 
 2\frac{\rn}{b_n}|R_n(\sigma_0)| \ge t \right]\1_{\te_0> t_n} \\
  & \le P_{\te_0}\left[
3 \frac{\rn}{b_n} \sup_{\sigma\in 2^{[n]}} |R_n(\sigma)| \ge t \right], 
\end{align*}
where the first three inequalities use identities obtained for $Z_n(\hat\sigma,X), Z_n(\sigma_0,X)$ above and the inequality obtained before the display, and
the last inequality uses $\delta(\hat\sigma,\sigma_0)- b_n \le 0$ and $\te_0\ge 0$.

Second, as $Z_n(\hat\sigma,X) < 0$ implies $Z_n(\hat\sigma,X) < - | Z_n(\sigma_0,X) |$ by definition of the maximum, 
\begin{align*}
\lefteqn{ P_{\te_0}\left[\rn|\hat\te-\te_0| \ge t , \,
Z_n(\hat\sigma,X) < 0\right]\1_{\te_0> t_n} } & \\
& \le P_{\te_0} \left[Z_n(\hat\sigma,X) < - | Z_n(\sigma_0,X) | \right]\1_{\te_0> t_n} \\
& \le P_{\te_0} \left[Z_n(\hat\sigma,X) < - |\te_0| b_n + |R_n(\sigma_0)| 
\right]\1_{\te_0> t_n} \\
& \le P_{\te_0} \left[\te_0 (\delta(\hat\sigma,\sigma_0)+b_n) <2 \sup_{\sigma\in 2^{[n]}} |R_n(\sigma)| \right]\1_{\te_0> t_n}\\
& \le P_{\te_0} \left[2 \sup_{\sigma\in 2^{[n]}} |R_n(\sigma)|> \frac{t_nb_n}{8} \right],
\end{align*}
where for the last inequality we have used the lower bound on $\delta$ obtained in Lemma \ref{lemde}. 

The case $\te_0<-t_n$ is treated in a symmetric way, by distinguishing the two cases 
$Z_n(\hat\sigma,X)<0$ and $Z_n(\hat\sigma,X)\ge 0$ respectively. To obtain a deviation bound for $\hat \te$, it is enough to study the supremum of 
the process 
$|R_n(\sigma)|$. For any given $\sigma$ and $y>0$, by Hoeffding's inequality, 
\[ P_{\te_0}[|R_n(\sigma)|>y]\le 2 \exp\{-2 y^2/b_n\}.\]
A union bound now leads to
\[ P\left[ \sup_{\sigma\in 2^{[n]}} |R_n(\sigma)| \ge y \right] \le 2^n
\exp\left\{-2 y^2/b_n\right\}\qquad\text{ for any }y>0. \]
This bound is smaller than $\exp\{-2n\}$ if one chooses $y=n^{3/2}$.
Combining the bounds obtained previously, and choosing $M_2$ above as $M_2=64$, one deduces
\[ P[\rn|\hat\te-\te_0|\ge  t] \le 6 e^{-2n}\qquad\text{ for any }t\ge 4.\]

\noindent The deviation bound in turn implies the bound in expectation
\begin{align*}
 E[n(\hat\te-\te_0)^2] & = E[n(\hat\te-\te_0)^2\1_{\rn|\hat\te-\te_0|\le  4}] + E[n(\hat\te-\te_0)^2\1_{\rn|\hat\te-\te_0|>  4}],\\
& \le 16 + 6n e^{-2n},
\end{align*}
where for the second term we have used $|\hat\te-\te_0|\le 1$, as $\Theta$ has diameter $1$. This concludes the proof of Theorem \ref{thmub} in the fixed design case.

In the random design case, one slightly updates the definition of $r_{ij}$. Here, the design is specified by $\vphi$, which is now random, but one 
can consider 
\[ \tilde r_{ij} = X_{ij} - E[X_{ij}\given \vphi]. \]
By definition, $E[ \tilde r_{ij} ]=E[ E[ \tilde r_{ij} \given \vphi]]=0$. Now one can follow the proof in the fixed design case by writing all statements conditionally on $\vphi$. As conditionally on $\vphi$ the variables $r_{ij}$ are independent and centered, the arguments leading to the 
various upper bounds remain unchanged. As the upper-bounds themselves do not depend on $\vphi$, the bounds also hold unconditionally.
\end{proof}

What remains to be shown is the bound on $\delta$ used  above:

\begin{lem} \label{lemde}
For any $\sigma_0,\sigma \in \Sigma$, with $\delta(\sigma,\sigma_0)$ defined in 
\eqref{deltas} and $b_n={n \choose 2}$,
\[ - \frac78 b_n \le \delta(\sigma,\sigma_0) \le b_n\qquad\text{ for any }\ n\ge 5.\]
\end{lem}
\begin{proof}
The upper bound corresponds to the number of terms in the sum. For the lower bound, 
 denote $\cC_1=\sigma_0^{-1}(\{1\})=\{i,\ \sigma_0(i)=1\}$. By symmetry, one can always assume $|\cC_1|\ge n/2$, otherwise one works with  $\cC_2=\sigma_0^{-1}(\{2\})$. 
The number $T_\sigma(\cC_1)$ of pairs $(i,j)\in\cC_1\times \cC_1$ for which $\sigma(i)\neq \sigma(j)$ is at most $2N_1(\sigma)N_2(\sigma)$, if $N_i(\sigma)=|\sigma^{-1}(\{i\})\cap \cC_1|$, $i=1,2$. This implies $T_\sigma(\cC_1)\le |\cC_1|^2/2$, using the inequality $p(q-p)\le q^2/4$, for 
any $0\le p\le q$. Thus the number of positive elements in the sum defining $\delta(\sigma,\sigma_0)$ is at least $(|\cC_1|^2 -  |\cC_1|^2/2 - |\cC_1|)/2$, where $|\cC_1|$ corresponds to the diagonal terms and the division by $2$ to the fact that the sum is restricted to $i<j$ only (note that the general term of the sum defining $\delta$ is symmetric in $i,j$). This is at least $|\cC_1|^2/8$ if $|\cC_1|\ge 3$. Hence,
\[ \delta(\sigma,\sigma_0)\ge -b_n+2\frac{|\cC_1|^2}{8},\]
which is the desired bound in view of $|\cC_1|^2\ge b_n/2$.
\end{proof}  

\subsection{Upper bound result, $k$ classes}
\label{sec-ub:k}

\begin{proof}[Proof of Theorem \ref{ubk}]   
First consider the fixed design case: As in the proof of Theorem \ref{thmub}, 
let $\te_0, \sigma_0$ denote the true values of $\te,\vphi$. The aim is to show that 
\[ \sup_{|\te_0|\le \kappa,\, \sigma_0\in\Sigma_e} E_{\te_0,\sigma_0}(\hat\te-\te)^2\le C \frac{k}{n}. \]
Let us denote by $Z^0$ and $\tilde Z$ the matrices of general terms $Z^0_{i,j}:=M_{\sigma_0(i)\sigma_0(j)}^{\te_0}$ and $\tilde Z_{ij}:=M^{\tilde \te}_{\tilde\sigma(i)\tilde\sigma(j)}$ respectively, with $\tilde\sigma,\tilde\te$ given by \eqref{ub-pre} and $i\neq j$. One interpretation of \eqref{ub-pre} if that the matrix $\tilde Z$  provides the best fit to the data $X_{ij}$ with respect to the squared $L^2$ loss, when optimising 
over $\Sigma_e\times \Theta_n$. 

As a first step, we show that $\tilde Z$ and $Z^0$ are close with high probability,  a result in the spirit of Gao et al \cite{chaogr15}, Theorem 2.1. This follows from Lemma \ref{lemzed} below, which states that $\|\tilde{Z}-Z^0\|^2\le Cn\log k$ with probability at least $1-e^{-n\log{k}}$, where $\|\cdot\|$ is the Frobenius norm.

In a second step,  denoting $S_I^0:=\sigma_0^{-1}(\{1,2\})$ and recalling from \eqref{sun-hat} that $\tilde S_I=\tilde\sigma^{-1}(\{1,2\})$, we show that $\tilde S_I$ is close to $S_I^0$. To do so,  one  separately bounds from below some terms from the quantity $\|Z^0 - \tilde Z\|^2= \sum_{i,j} (Z^0_{i,j} - \tilde Z_{i,j})^2$, recalling that one extends the 
$Z$ matrices by symmetry and sets the diagonal to $0$. First, using the definitions, \eqref{techm}, $|\te|\le \kappa$, and $n\ge 2$,
\[   \sum_{i,j\,\in\, S_I^0\,\setminus\, \tilde S_I} (Z^0_{ij} - \tilde Z_{ij})^2  
\ge \frac{\kappa^2}{2} | S_I^0\,\setminus\, \tilde S_I |^2, \]
if $ | S_I^0\,\setminus\, \tilde S_I |\ge 2$ (otherwise the inequality below holds trivially), as well as 
\[ \sum_{i\,\in\, S_I^0\,\cap\, \tilde S_I\,,\, j\,\in\, S_I^0\,\setminus\, \tilde S_I} (Z^0_{ij} - \tilde Z_{ij})^2 \ge  
\frac{\kappa^2}{2}\, |S_I^0\,\cap \,\tilde S_I|\,| S_I^0\,\setminus\, \tilde S_I |, \]
and, with $a\wedge b=\min(a,b)$, if $|S_I^0\,\cap\, \tilde S_I|\ge 2$,
\[  \sum_{i,j\,\in\, S_I^0\,\cap\, \tilde S_I} (Z^0_{ij} - \tilde Z_{ij})^2 \ge \frac12\left[ (\te_0-\tilde\te)^2 \wedge (\te_0+\tilde\te)^2\right] |S_I^0\,\cap \,\tilde S_I|^2.\]
The previous bound on $\|\tilde Z - Z^0\|^2$ implies that
\[ \max\left(\,| S_I^0\,\setminus\, \tilde S_I |^2\,,\,|S_I^0\,\cap \,\tilde S_I|\,| S_I^0\,\setminus\, \tilde S_I |\,\right) \le C\kappa^{-2} n\log{k}\ \text{ with high probability.}\]
It now follows from \eqref{techc} that for any $\delta>0$, one has $C\kappa^{-2} n\log{k}\le \delta n^2/k^2$, provided $d$ in \eqref{techc} is small enough.  So for small $d$, as $|S_I^0|=|S_I^0\cap \tilde S_I|+ |S_I^0\setminus \tilde S_I |$, and as by assumption $\sigma_0\in\Sigma_e$ so that 
$|S_I^0|\asymp n/k$, one deduces $|S_I^0 \cap \tilde S_I|\geqa n/k$.
From the bound on $\|\tilde Z-Z^0\|^2$, it follows that 
\[ (\te_0-\tilde\te)^2\wedge (\te_0+\tilde\te)^2 \leqa \frac{n\log{k}}{n^2 k^{-2}} \leqa
\frac{k^2\log{k}}{n}\qquad\text{ with high probability.} \]
By \eqref{techc} this shows that $\tilde\te$ is close to either $\te_0$ or $-\te_0$ up to $k(\log{n}/n)^{1/2}=:\rho\le \kappa/2$. Now for any $i,j$ in $\tilde S_I\setminus S_I^0$, if $\tilde Z_{ij}=\frac12+\tilde{\te}$ and $\tilde\te$ is close to $\te_0$ (the cases where $\tilde Z_{ij}=\frac12- \tilde{\te}$ or $\tilde\te$  is close to $-\te_0$ are treated similarly), setting $Z_{ij}^0=c_{ij}$ and using \eqref{techc}, with $a_0=1/2$,
\begin{align*}
 |Z_{ij}^0 - \tilde{Z}_{ij}| & =|c_{ij}- (a_0+\tilde\te)| =|c_{ij} - (a_0+\te^0) - (\tilde\te - \te^0)| \\
 & \ge |c_{ij} - (a_0+\te^0)| - |\tilde\te-\te_0|\ge \kappa - \rho \ge \kappa/2.
\end{align*}
Therefore,
\[\kappa^2 | \tilde S_I \setminus S_I^0|^2 \leqa  \sum_{i,j\,\in\, \tilde 
S_I \setminus S_I^0} (Z^0_{ij} - \tilde Z_{ij})^2 \leqa n\log{k}\qquad\text{ with high probability} \]
and
\[ \kappa^2 |S_I^0\,\cap \,\tilde S_I|\,| \tilde S_I\,\setminus\, S_I^0 | 
\leqa  \sum_{\substack{i\in\, S_I^0\,\cap\, \tilde S_I\\ j\in\, \tilde S_I\,\setminus\, S_I^0}} (Z^0_{ij} - \tilde Z_{ij})^2 \leqa n\log{k}\qquad\text{ with high probability.} \]
Combining the previous bounds on cardinalities and denoting $A\ \Delta\ B:= (A\setminus B)\,\cup\, (B\setminus A)$ for two sets $A$  and $B$, one obtains
\[ |\tilde S_I \ \Delta\ S_I^0|^2 + |\tilde S_I \ \Delta\ S_I^0||S_I^0\,\cap \,\tilde S_I|
\leqa \kappa^{-2}n\log{k}\qquad\text{ with high probability,}\]
which in turn implies that
\begin{equation} \label{trc}
 |\{\tilde S_I \times \tilde S_I\} \ \Delta\   \{S_I^0 \times S_I^0\}| \leqa \kappa^{-2}n\log{k}\qquad\text{ with high probability.}
\end{equation} 

In a third and last step, we follow the proof of Theorem \ref{thmub}. Let 
$\hat \sigma=\hat\sigma_I$ be the mapping in \eqref{sigi}. It is a map $\tilde{S}_I\to \{1,2\}$. Let $\bar{\sigma}$ be the mapping $S_I^0\to \{1,2\}$ that coincides with $\hat \sigma$ on $\tilde S_I \cap S_I^0$ and with $\sigma_0$ on $S_I^0\setminus \tilde S_I$. By definition we have, with 
$\Delta_n:=\kappa^{-2}n\log{k}$,
\begin{align*}
Z_n(\hat\sigma,\tilde S_I,X) & = \sum_{i<j, i,j\in \tilde S_I}
(-1)^{\1_{\hat\sigma(i)=\hat\sigma(j)}}\left(\frac12-X_{ij}\right) \\
& = \sum_{i<j, i,j\in  S^0_I}
(-1)^{\1_{\bar\sigma(i)=\bar\sigma(j)}}\left(\frac12-X_{ij}\right) + O(\Delta_n)\\
& = \delta(\bar\sigma,\sigma_0)\te_0 + R_n(\bar{\sigma}) + O(\Delta_n),
\end{align*}
where for the second identity we have used that the $X_{ij}$s are bounded 
by $1$ and \eqref{trc}, and $R_n(\sigma)$ is defined as in the proof of Theorem \ref{thmub}. Similarly, denoting $n_k={|S_I^0| \choose 2}$  
and $\tilde{n}_k={|\tilde S_I| \choose 2}$, we have 
\begin{align*}
Z_n(\sigma_0,\tilde{S}_I,X) & = n_k \te_0+O(\Delta_n) + R_n(\sigma_0) 
 = \tilde{n}_k\te_0+O(\Delta_n) + R_n(\sigma_0),
\end{align*}
with high probability, since \eqref{trc} implies $|\tilde{n}_k-n_k|=O(\Delta_n)$  using that $||A|-|B||\le |A\Delta B|$ for two sets $A, B$. Also, since $\hat\te = Z_n(\hat\sigma,\tilde{S}_I,X)/\tilde{n}_k$ and $|\hat\te|\le 1/2$, by the same argument we have
\begin{equation} \label{htequiv}
\hat\te = \frac{Z_n(\hat\sigma,\tilde{S}_I,X)}{n_k} + \frac{O(\Delta_n)}{n_k}.
\end{equation}

Let $v_k$ and $t_k$ be two sequences depending on $n$ and $k$ whose specific values are determined below (see the last paragraph of the proof). 
If $Z_n(\hat\sigma,\tilde{S}_I,X)\ge 0$, then $Z_n(\hat\sigma,\tilde{S}_I,X)\ge |Z_n(\sigma_0,\tilde{S}_I,X)|\ge Z_n(\sigma_0,\tilde{S}_I,X)$.
Let $\Sigma_0$ be the set of all maps $S_0\to\{1,2\}$. In the following inequalities we repeatedly use the fact that the normalisation $\tilde n_k$ in the definition of $\hat\te$ can be replaced by $n_k$ up to a factor $O(\Delta_n)/n_k$, see \eqref{htequiv},
\begin{align*}
\lefteqn{ P_{\te_0}\left[v_k|\hat\te-\te_0| \ge t , 
Z_n(\hat\sigma,\tilde{S}_I,X) \ge 0\right]\1_{\te_0> t_k} }& \\
& \le P_{\te_0}\Big[\frac{v_k}{n_k}| Z_n(\hat\sigma,\tilde{S}_I,X)- Z_n(\sigma_0,\tilde{S}_I,X) + O(\Delta_n) -
 R_n(\sigma_0)| \ge t , \\  
& \qquad\qquad Z_n(\hat\sigma,\tilde{S}_I,X) \ge 0\Big]\1_{\te_0> t_k} \\
& \le P_{\te_0}\left[\frac{v_k}{n_k}\left[ Z_n(\hat\sigma,\tilde{S}_I,X)- 
Z_n(\sigma_0,\tilde{S}_I,X)\right] +
 \frac{v_k}{n_k}\left[|R_n(\sigma_0)|+O(\Delta_n)\right] \ge t \right]\1_{\te_0> t_k} \\
 & \le P_{\te_0}\left[\frac{v_k}{n_k} \te_0 (\delta(\bar\sigma,\sigma_0) - n_k) 
 + \frac{v_k}{n_k}\left[R_n(\hat\sigma)+|R_n(\sigma_0)|+O(\Delta_n)\right]  \ge t \right]\1_{\te_0> t_k} \\
  & \le P_{\te_0}\left[
2 \frac{v_k}{n_k} \sup_{\sigma\in \Sigma_0} |R_n(\sigma)| + \frac{v_k}{n_k} O(\Delta_n) \ge t \right], 
\end{align*}
where the last inequality uses $\delta(\bar\sigma,\sigma_0)- n_k \le 0$, see Lemma \ref{lemde} with $n_k$ in place of $b_n$, and $\te_0\ge 0$.

Second, as $Z_n(\hat\sigma,\tilde{S}_I,X) < 0$ implies $Z_n(\hat\sigma,\tilde{S}_I,X) < - | Z_n(\sigma_0,\tilde{S}_I,X) |$ by definition of $\hat\sigma$, 
\begin{align*}
\lefteqn{ P_{\te_0}\left[ Z_n(\hat\sigma,\tilde{S}_I,X) < 0\right]\1_{\te_0> t_k} } & \\
& \le P_{\te_0} \left[Z_n(\hat\sigma,\tilde{S}_I,X) < - | Z_n(\sigma_0,\tilde{S}_I,X) | \right]\1_{\te_0> t_k} \\
& \le P_{\te_0} \left[Z_n(\hat\sigma,\tilde{S}_I,X)  < - |\te_0| n_k + O(\Delta_n) 
+ |R_n(\sigma_0)| \right]\1_{\te_0> t_k} \\
& \le P_{\te_0} \left[\te_0 (\delta(\bar\sigma,\sigma_0)+n_k) - 2O(\Delta_n) <2 \sup_{\sigma\in\Sigma} |R_n(\sigma)| \right]\1_{\te_0> t_k}\\
& \le P_{\te_0} \left[2 \sup_{\sigma\in\Sigma_0} |R_n(\sigma)|> \frac{t_k 
n_k}{8} 
- O(\Delta_n) \right], 
\end{align*}
where for the last inequality we have used the first inequality of Lemma \ref{lemde}. 
Also,
\begin{align*}
\lefteqn{P_{\te_0}\left[v_k |\hat\te-\te_0|\ge t\right] \1_{|\te_0|\le t_k} }\\ 
& \le \quad P_{\te_0}\left[
\frac{v_k}{n_k}| \te_0 | |\delta(\bar\sigma,\sigma_0) - n_k|
 + \frac{v_k}{n_k}\left[O(\Delta_n)+ |R_n(\bar\sigma)| + 
 |R_n(\sigma_0)|\right]  \ge t \right]\1_{|\te_0|\le t_k}\\
 & \le \quad  P_{\te_0}\left[2\frac{v_k}{n_k}\sup_{\sigma\in\Sigma_0} |R_n(\sigma)| 
 + \frac{v_k}{n_k}O(\Delta_n) \ge t-v_kt_k \right]. 
\end{align*}
By the same argument as in the proof of Theorem \ref{thmub}, the supremum 

$\sup_{\sigma\in\Sigma_0} |R_n(\sigma)|$ is of the order $|\Sigma_0|^{3/2}\asymp (n/k)^{3/2}$, by definition of $\Sigma_e$. Recall that $n_k\asymp 
(n/k)^2$. 
Set $v_k=\sqrt{n/k}$ and $t_k=Dv_k^{-1}$, with $D$ a large enough constant. Assumption \eqref{techc} ensures that $\Delta_n=\kappa^{-2} n\log{k}=O((n/k)^{3/2})$. Hence by taking $t$ a large enough constant, one obtains that the last three displays are bounded above by $e^{-Cn/k}$, which concludes the proof in the fixed design case, proceeding as in  the proof of Theorem \ref{thmub} to get the final bound in expectation.

The proof in the random design case is obtained by first deriving the results conditionally on $\vphi$ and then integrating out $\vphi$, as we did 
in the proof of Theorem \ref{thmub}. The first part is almost identical to the fixed design case: one only needs to note that one can restrict to mappings $\vphi$ that belong to, essentially,  $\Sigma_e$.  Denote by $\Sigma_e'$ the subset of those ${\sigma\in\Sigma_e}$ satisfying $||\sigma^{-1}(j)|-\frac{n}{k}|<\frac{n}{2k}$. Then
\begin{align*}
 E_{\te_0}[ (\hat\te-\te_0)^2 ] & =  E_{\te_0}[ (\hat\te-\te_0)^2 \1\{ \vphi \in \Sigma_e' \}]
 + E_{\te_0}[ (\hat\te-\te_0)^2 \1\{ \vphi \in \Sigma_e' \}] \\
 & = E_{\te_0}\left[ \, E_{\te_0}[(\hat\te-\te_0)^2 \given\vphi]\, \1\{ 
\vphi \in \Sigma_e' \}\right] + P_{\te_0}[\vphi\notin \Sigma_e'].
\end{align*} 
For the first term, one can apply the arguments above in fixed design, while for the second an application of Bernstein's inequality gives  
\[ P_{\te_0}\left[\vphi\notin \Sigma_e'\right]\le kP\left[ |\text{Bin}(n,k) - \frac{n}{k}| >   \frac{n}{2k}
\right] \le 2k e^{-\frac{n}{10k}}. \]
By \eqref{techc}, $k^3\le Cn$ holds for $C$ large enough---note that $\kappa$ must be smaller than $1$, as the entries of $M$ are in $[0,1]$. One deduces $k e^{-\frac{n}{10k}}\le \frac{k}{n} ne^{-dn^{2/3}}\le C\frac{k}{n}$, for $d$ small enough and $C$ large enough, so the quadratic risk is at most $Ck/n$ in this case as well.
\end{proof}
\vspace{.3cm}

\begin{lem}\label{lemzed}
Let $Z^0_{i,j}:=M_{\sigma_0(i)\sigma_0(j)}^{\te_0}$ and $\tilde Z_{ij}:=M^{\tilde \te}_{\tilde\sigma(i)\tilde\sigma(j)}$, with $\tilde\sigma,\tilde\te$ given by \eqref{ub-pre}. Let $\|\cdot\|$ denote the matrix Frobenius norm. With probability at least $1-e^{-cn\log{k}}$, 
\[ \|\tilde{Z}-Z^0\|^2 
\leqa Cn\log{k}.\] 
\end{lem}
\vspace{.2cm}

\begin{proof}[Proof of Lemma \ref{lemzed}] Let $\te_1$ denote the element 
of $\Theta_n$ closest to $\te_0$, so that $|\te_0-\te_1|\le n^{-2}$. Let $Z^1$ be the matrix given by $Z^1_{ij}:= M^{\te_1}_{\sigma_0(i)\sigma_0(j)}$. 
By definition, $\|X-\tilde Z\|^2 \le \|X- Z^1\|^2$ and hence 
$\|Z^1 - \tilde Z\|^2 + 2\psg X-Z^1, Z^1-\tilde Z\psd \le 0$, so
\begin{align*}
 \|Z^1 - \tilde Z\|^2 & \le  2\|Z^1-\tilde Z\| \sup_{\te\in\Theta_n,\ \sigma\in\Sigma_e} |\psg X-Z^1,  \frac{M^\te_{\sigma(i)\sigma(j)}-Z^1}{\|M^\te_{\sigma(i)\sigma(j)}-Z^1\|}\psd| \\
& \le 2 \|Z^1-\tilde Z\|
\Big[\sup_{\te\in\Theta_n,\ \sigma\in\Sigma_e} |T_n(\sigma,\theta)|
+ \|Z^0-Z^1\|\Big],
\end{align*} 
where we denote $T_n(\sigma,\te):= \psg X-Z^0,  (M^\te_{\sigma(i)\sigma(j)}-Z^1)
/\|M^\te_{\sigma(i)\sigma(j)}-Z^1\|\psd$. As elements of the matrix $X-Z^0$ are between $-1$ and $1$, we note that $T_n(\sigma,\te)$ is of the form $\sum_{l} \mu_l\veps_l$, where $\veps_l\in[-1,1]$ are independent, and $\sum_{l}\mu_l^2=1$. So using Hoeffding's inequality, for any $t>0$,
\[ P[ |T_n(\sigma,\te)|>t] \le 2 \exp\{-t^2/2\}. \]
The cardinality of the set $\Theta_n\times \Sigma_e$ is bounded above by $(2n^2+1) k^n\leqa k^{Cn}$. A union bound then shows that, with probability at least $1-e^{-cn\log k}$,
\begin{equation*}   
\sup_{\te\in\Theta,\ \sigma\in\Sigma_e} |T_n(\sigma,\te)| \le C\sqrt{n\log{k}}.
\end{equation*}
Inserting this back into the previous inequality on $\|Z^1-\tilde{Z}\|^2$ 
leads to $\|Z^1-\tilde{Z}\|\leq C\sqrt{n\log k}+\|Z^0-Z^1\|$ with probability at least $1-e^{-cn\log{k}}$. As $\|Z^0- Z^1\|^2 \le Cn^2/n^2\le C$, the triangle inequality leads to the result.
\end{proof}

\section{Remaining proofs: lower and upper bounds in the sparse case}
\label{sec:ub_spec}

We begin with a brief overview of the proof techniques: For the lower bounds in the sparse setting (Theorems \ref{thm-twocl-spa} and 
\ref{thm-kcl-spa}), proofs are very similar to the dense case, and it suffices to track the dependence on the sparsity parameter $\alpha_n$. To upper-bound the convergence rate of spectral estimates (Theorems \ref{thm_ubspec} and \ref{thm-twocl-spa}), we use the fact that $\theta$ can be estimated from the largest absolute eigenvalue of the (translated) adjacency matrix. The latter can in turn be estimated empirically. 
For the proofs of the upper bounds for $k$ classes, we show that with high probability it is possible to recover the true aggregated labels, where aggregation means that classes $1$ and $2$, corresponding to the `hard submodel' are merged (this is the `$g$ map' introduced above in the second paragraph of Appendix \ref{sec:spec:k}). To do so, one adapt techniques introduced by \citet{Lei:Rinaldo:2015:1}, and \citet{Lei:Zhu:2017:1} and show that their results still hold `under small perturbations', as explained in more details below. Once the true aggregated labels of classes $1$ and $2$ are obtained, it suffices to apply the (already derived) result for the case $k=2$.

\subsection{Proofs for the two-class case}

\begin{proof}[Proof of the lower bound in Theorem \ref{thm-twocl-spa}]
One proceeds in the same way as in the proof of Theorem \ref{thm-twocl}
with $a_{ij}(\vphi)$  replaced by 
$b_{ij}(\vphi):=Q_{\vphi(i)\vphi(j)}^\theta(\al_n)$, $i<j$. 
If $\la_1=\text{Be}(b_{ij}(\vphi))$, $\la_2=\text{Be}(b_{ij}(\psi))$, 
$\mu=\text{Be}(\al_n/2)$, we now have 
\[ \ta_{ij}(\vphi,\psi) := \int (\frac{d\la_1}{d\mu}-1)(\frac{d\la_2}{d\mu}-1)d\mu 
= \frac{(2b_{ij}(\vphi)-\al_n)(2 b_{ij}(\psi)-\al_n)}{\al_n(2-\al_n)}. \]
This leads to, with $\eta_i=\1_{\vphi(i)=1}-\1_{\vphi(i)=2}$ and $\eta_i'=\1_{\psi(i)=1}-\1_{\psi(i)=2}$,
\[ \ta_{i,j}(\vphi,\psi)=\frac{4\al_n\theta^2\eta_i\eta_j\eta_i'\eta_j'}{2-\al_n}. \]
By the same argument as in the proof of Theorem \ref{thm-twocl},
it is enough to solve
\[ \frac{n\al_n\te^2}{2-\al_n} = C \]
for $\te$, where $C$ is a universal positive small enough constant, under the constraint that $|\te|\le 1/2$. This leads to take $\te^2$ equal up to a constant to $(n\al_n)^{-1}\wedge 1$ and the proof is complete.
\end{proof}

\begin{proof}[Proof of the upper bound in Theorem \ref{thm-twocl-spa}] We write the proof directly in the possibly sparse setting. 
Let us first consider the fixed design case, where $\vphi$ is non-random. 
Let  $\|\,.\,\|_{Sp}$ denote the spectral norm of a matrix (for a symmetric matrix $\Delta$, $\|\,.\,\|_{Sp}=\max(|\la_1(\,.\,)|,|\la_n(\,.\,)|)$, so  $|\la_1^a(\Delta)|=\|\Delta\|_{Sp}$).
By \citep[][Theorem 5.2]{Lei:Rinaldo:2015:1}, we have that for any $r>0$, 
there exists a $C=C(r,c_0)>0$ such that
\[ \|X-E[X]\|_{Sp} \le C\sqrt{n\al_n}\;,\]
with probability at least $1-n^{-r}$. From this one deduces that
$\|\Delta_n-E[\Delta_n]\|_{Sp}\le C\sqrt{n\al_n}$.
The eigenvalues of $\Delta$ and those of ${\Delta_n-E[\Delta_n]}$ and ${E[\Delta_n]}$ can be related to each other by a Weyl-type inequality as 
\[ |\la_i(\Delta_n)-\la_i(E\Delta_n)| \le \|\Delta_n-E\Delta_n\|_{Sp}\;, \]
for any $1\le i\le n$, see e.g. \citep[][eq. (1.64)]{Tao:2012}. Suppose for now that $\te\ge 0$. In this case $\la_1(E\Delta_n)=(n-1)\al_n\te$ and $\la_n(E\Delta_n)=0$, which by the previous inequality implies, with 
high probability,
\[ \left|\frac{\la_1(\Delta_n)}{\al_n(n-1)}-\te\right| \le \frac{C}{\sqrt{n\al_n}},\quad \left|\frac{\la_n(\Delta_n)}{\al_n(n-1)}\right|\le \frac{C}{\sqrt{n\al_n}}.  \]
Now if $\te>2C/\sqrt{n\al_n}$, using the first inequality we have $\la_1(\Delta_n)/\{\al_n(n-1)\}>C/\sqrt{n\al_n}$ and $\tilde{\la}_1=\la_1(\Delta_n)$ follows from the second inequality, which means $|\hat\te_n-\te|\le C/\sqrt{n\al_n}$. If $\te\le2C/\sqrt{n\al_n}$, the triangle inequality and the second inequality imply $|\la_n(\Delta_n)-\te|\le 3C/\sqrt{n\al_n}$, which combined with the first inequality gives  $|\hat\te_n-\te|\le 3C/\sqrt{n\al_n}$. So, for $\te\ge0$, in all cases $|\hat\te_n-\te|\le 3C/\sqrt{n\al_n}$ with high probability. The case $\te<0$ is treated similarly. 
In the random design setting, one can argue conditionally on $\vphi$, and 
then note that both the obtained bounds and the in-probability statements 
do not depend on $\vphi$, which gives the result in this setting as well.
\end{proof}

\subsection{Proofs for the general case}

\begin{proof}[Proof of the lower bound in Theorem \ref{thm-kcl-spa}]
The proof is similar to that of Theorem \ref{thm-kcl}, where one now uses 
the sparse lower bound for two classes of Theorem \ref{thm-twocl-spa} instead of Theorem \ref{thm-twocl}, and is thus omitted.
\end{proof}

\begin{proof}[Proof of the upper bound in Theorem \ref{thm-kcl-spa}] 
We show that the proof approach used by \cite{Lei:Zhu:2017:1} to establish their Theorem 2 can be adapted to our problem. More precisely, it is amenable to a perturbation of the true matrix $M$ of connection probabilities: We show that, for a graph generated by  model \eqref{zte} with a sufficiently small value of $\theta$, the {\tt V-Clust} algorithm with $K=k-1$ classes recovers the aggregated labelling defined by $g$ above with high probability. We do the proof in the possibly sparse situation, thereby also proving the upper-bound in Theorem \ref{thm-kcl-spa}.

There are three steps. First, we show that the initial label recovery algorithm $\cS$ of \cite{Lei:Rinaldo:2015:1} recovers most of the labels correctly, and control the error. Second, we show that the scheme of proof of \cite{Lei:Zhu:2017:1} carries over to the problem of recovering the aggregated clustering up to label permutation. Finally, using assumption (A3) one can recover the aggregated class $1$ with high probability, and restricting to nodes with label in that class we can apply the spectral method $\cS_2$ of the case $k=2$.\\[.5em]

\noindent{\em First step (Perturbed spectral method of Lei and Rinaldo).}
\[ M^\te = M^0 + \te R, \quad \text{with}\qquad R =
\begin{bmatrix}
1 & -1 & 0_{2,k-2}\\
-1 & 1 & 0_{2,k-2}\\
0_{k-2,1}& 0_{k-2,1}& 0_{k-2,k-2}
\end{bmatrix}
\]
The matrix $M^0$ (i.e. $M^\te$ with $\te=0$) can be transformed into the matrix $N$ above by removing the first line and then the first column.

Let $X$ be the matrix $(X_{ij})$. Since the relevant design is fixed, there exists a binary $n\times k$ matrix $T$, with a single 1 in each row, for which we have ${E[X]=TM^{\te}(\al_n)T^t+D}$, where $D=-\text{Diag}(TM^{\te}(\al_n)T^t)$ is a diagonal matrix with entries bounded by $\al_n$. Lei and Rinaldo call $T$ a membership matrix. It can be rewritten in terms of $N$, using the relation between $M^0$ and $N$ noted above: for a $n\times 
K$ membership matrix $S$ and $E[X]$ the expected value of $X$,
\[ E[X] = \al_nS N S^t + \al_n \te T R T^t + D. \]
Now we can follow Lei and Rinaldo's analysis of simple spectral clustering with $K=k-1$ and the expectation matrix $S N S^t$; one only needs to show that, despite the perturbation $\te T R T^t$, the argument still holds. Intuitively, this is guaranteed by the assumption that $\te$ is small 
enough, which ensures that the spectrum of the perturbation $\te T R T^t$ 
does not interact much with that of $S N S^t$. More precisely, we decompose $X$ as
\[ X = P + W,\qquad \text{with}\quad P := \al_nS N S^t, \quad W := \al_n\te T R T^t + D+ X-E[X]. \]
Following the proof of Theorem 3.1 of \cite{Lei:Rinaldo:2015:1}, the pair 
$(S,N)$ parametrises a SBM with $K=k-1$ classes and $N$ is full rank. By their Lemma 2.1, the eigendecomposition of $P=S(\al_nN)S^t$ can be written $P=UDU^t$, where $U$ is the matrix of the $K$ leading eigenvectors of $P$, and one can write $U=S\xi$, for some matrix $\xi\in\RR^{K\times K}$ with orthogonal rows (and $\|\xi_{k*}-\xi_{l*}\|^2=n_k^{-1}+n_l^{-1}$). It also follows from the proof of that Lemma that if $\ga_n$ denotes the smallest absolute nonzero eigenvalue of $P$, we have $\ga_n=n_{min}\al_n\la(K)$, with $n_{min}$ the cardinality of the smallest class, here of order $n/k$ using that classes are balanced, and $\la(K)$ the smallest absolute eigenvalue of $N$.

By Lemma 5.1 of  \cite{Lei:Rinaldo:2015:1}, one can control the distance between the leading eigenspaces of $X$ and $P$ (for the first $K$ non-zero eigenvalues) in terms of the spectral norm of $W$. The assumptions of that Lemma are fulfilled with $P$ here of rank $K=k-1$ and of smallest nonzero singular value $\ga_n$. If $\hat U\in \RR^{n\times K}$ is the matrix of the $K$ leading eigenvectors of $X$ (and $U$ the one for $P$, as above),  there exists a $K\times K$ orthogonal matrix $Q$ such that, with $\|\cdot\|_K$ and $\|\cdot\|_{Sp}$ the Frobenius and spectral norms respectively,
\[ \| \hat{U} - UQ\|_F \le \frac{2\sqrt{2K}}{\ga_n} \|X-P\|_{Sp}. \]
By the triangle inequality, the spectral norm $\|X-P\|_{Sp}$ is in turn bounded by
\[ \| X - P \|_{Sp} \le \| X-E[X] \|_{Sp} + \al_n|\te| \| TRT^t \|_{Sp}. \]
The matrix $R$ can be written $R=uu^t$, where $u^t$ is the row $(1 \,-1 
\ 0 \ldots 0)$ of length $k$. In particular, $R$ is of rank $1$, and $\|TRT^t\|_{Sp}=\|Tu\|_2^2$ (a nonzero eigenvector is $Tu$). By construction, $\|Tu\|_2^2=n_1+n_2$, the number of elements of classes $1$ and $2$, 
so that
$\|TRT^t\|_{Sp}\le Cn/K$. Also, $\|D\|_{Sp}\le \al_n$ since $D$ is diagonal with terms bounded by $\al_n$. By Theorem 5.2 of  \cite{Lei:Rinaldo:2015:1}, the norm $\|X-E[X]\|$ is, with probability at least $1-1/n^2$, no larger than $C\sqrt{n\al_n}$, for a sufficiently large constant $C$. Gathering the last bounds and using $\al_n\leqa \sqrt{n\al_n}$, one obtains $\|X-P\|_{Sp}\le C(\sqrt{n\al_n}+|\te| n\al_n/K)$.

On the other hand, following \citet{Lei:Rinaldo:2015:1}, one can perform an $(1+\veps)-$approxi\-mate $k$-means clustering on the rows of $\hat U$: Application of their Lemma 5.3 to the matrices $\hat U$ and $UQ$ shows the approximate $k$-means solution is a pair $(\hat{S},\hat{\xi})$, where 
$\hat S$ a membership matrix, $\hat{\xi}$ a $K\times K$ matrix, and $\hat{S}\hat{\xi}$ is an approximate least-squares fit to $\hat U$. Moreover, the estimated membership $\hat S$ coincides with $S$ up to label permutation, except on sets ${S_1,\ldots,S_K}$ that are characterized as follows: 
Recall that $\psi$ is the `true' labelling obtained by merging the original classes 1 and 2 of nodes. Each set $S_{j}\subset \psi^{-1}(j)$ satisfies
\[ \frac{1}{n_{min}}\sum_{j=1}^{K} |S_j| \le 4(4+2\veps)\|\hat U - UQ\|_F^2, \]
whenever
\begin{equation} \label{condu}
(16+8\veps)\|\hat U - UQ\|_F^2<1\;.
\end{equation}
This implies, using the previous bounds and
$\ga_n=n_{min}\al_n\la(K)\geqa (n/K)\al_n\la(K)$, that
\begin{align*}
 \frac1n \sum_{j=1}^K |S_j| & \le (16+8\veps)\pi_0  \frac{\|\hat U - UQ\|_F^2}{K} \\
 & \le (16+8\veps)\pi_0  \frac{8}{\ga_n^2}2C^2(n\al_n+\te^2\al_n^2\frac{n^2}{K^2})\\
 & \le \frac{CK^2}{\la(K)^2}(\frac{1}{n\al_n}+\frac{\te^2}{K^2}),
\end{align*}
provided, for some suitably small constant $c>0$, with $\la=\la(K)$,
\[  K^3\frac{1}{n\al_n\la^2} + K\frac{\te^2}{\la^2}<c. \]
The first summand coincides with the condition in  \cite{Lei:Rinaldo:2015:1}. The second term accounts for the perturbation induced by $\al_n\te R$. Provided that
\begin{equation} \label{lz_cond1}
 \frac{K^3}{\al_n\la^2}<cn/2 \quad\text{ and }\quad \te^2<\frac{c\la^2}{2K},
\end{equation}
the simple spectral clustering algorithm has recovery error at most $n/f(n\al_n,K)$, with
\begin{equation} \label{lz_error}
 f(n\al_n,K) = C (\frac{n\al_n\la^2}{K^2} \wedge \frac{\la^2}{\te^2}).
\end{equation}
The conditions on $n,\al_n, \la, K, \te$ permit this quantity to be chosen suitably large. This means that, with high probability, Step 1 of the algorithm with $K=k-1$ recovers a sufficiently large proportions of the labels of $N$, up to label permutation.\\[.5em]

\noindent{\em Second step (Lei and Zhu's exact label recovery method via sample splitting).} We can now use the method introduced by \citet{Lei:Zhu:2017:1}: using a first rough estimate of the labels, one can refine it to an exact label recovery with high probability, provided $f(n\al_n,K)$ is large enough in terms of a certain function of $K$. The recovered  labels are those of the original classes $3,4,\ldots,k$, and of the aggregated class containing classes $1$ and $2$. To verify that the proof of Lei and Zhu generalizes to the perturbed cased, it suffices to note that the distortion of $E[X_{ij}]$ for $i,j\in \psi^{-1}(\{1,2\})$ from $1/2$ to $1/2\pm\te$ does not interfere with the bounds of the proof of Theorem 2 in \cite{Lei:Zhu:2017:1}. The sample splitting algorithm of Lei and Zhu involves two subroutines called ${\tt CrossClust}$ and ${\tt Merge}$.
The mean of $X_{ij}$ enters in the proof of that result via two applications of Bernstein's inequality, in the proofs of Lemma 6 (which implies the consistency of {\tt CrossClust} via Lemma 3) and Lemma 7 (consistency of {\tt Merge}) of \citet{Lei:Zhu:2017:1}.

We impose the assumptions of \citet{Lei:Zhu:2017:1}, Theorem 2, on $K$ and $f(\al_nn/2,K)$: one needs $f(n\al_n/2,K)\ga(K)\ge CK^{2.5}$ and the inequalities $\al_n\ge CK^3\log{n}/(\ga^2(K)n)$ and $Cn\ge K^3$. The last two conditions are 
implied by (B2) (respectively (A2) in the dense case). By  \eqref{lz_error}, the first one is satisfied if
\[ C\frac{n\al_n\la^2}{K^2} \ga(K) \ge CK^{2.5},\qquad\text{and}\qquad \frac{\la^2}{\te^2}\ga(K) \ge CK^{2.5}.\]
Again, the first inequality holds by (B2). The second inequality asks for 
$\te^2<C\la^2\ga(K)/K^{2.5}$, which is guaranteed by \eqref{def_tk}.

The parameter $\te$ affects the proof of Lemma 3 of \citet{Lei:Zhu:2017:1} as follows. The proof relies on bounding three terms $T_1, T_2$ and $T_3$ via Lemma 6 in the Appendix of their paper. To be able to apply Bernstein's inequality on $T_1$, one needs
\[ (n_1+n_2)\al_n|\te|<\frac14 \pi_0\ga(K)\frac{n\al_n}{30K^{3/2}},\]
where $\pi_0=n_{min}/(n/K)\geqa 1$ in the notation of \cite{Lei:Zhu:2017:1}, Definition 1. To bound $T_2$, one needs
\[ (n_1+n_2)\al_n|\te|<\frac14  \frac{\pi_0^2\ga(K)n\al_nf(n\al_n/2,K)}{30K^{5/2}}\;, \]
while, similarly, to bound $T_3$ one needs
\[ (n_1+n_2)\al_n|\te|<\frac14 \pi_0\ga(K)\frac{n\al_n}{15K^{3/2}}. \]
On the other hand, consistency of {\tt Merge} with $V=2$ requires
\[ (n_1+n_2)\al_n|\te|<\frac14 \frac{\al_n\pi_0^2\ga(K)n^2}{160K^3}. \]
The condition required for $T_1$ implies the remaining ones: The one required for $T_2$ is weaker, since
$f(n\al_n,K)>C'K$ for some constant $C'>0$, by \eqref{lz_cond1}--\eqref{lz_error}. The condition for {\tt Merge} is also weaker up to constants, provided that $n\geqa  K^{3/2}$, which is satisfied under our conditions.
To obtain the above Bernstein's inequalities, one thus needs
\[ |\te| < c_1\frac{\ga(K)}{\sqrt{K}}. \]
It follows from the spectral decomposition of $N$ that $\la(K)\le \ga(K)/\sqrt{2}$. By combining with \eqref{lz_cond1}, we see that it is enough that $\te$ satisfies $|\te|\le C\la/\sqrt{K}$, which was already required above. Finally, to see that the last inequality is satisfied, one notes that it is implied by \eqref{def_tk}, using that $\ga(K)\le \sqrt{K}$.  We 
have just proved that the exact recovery from \cite{Lei:Zhu:2017:1} also holds here.
\\[.5em]

\noindent{\em Third step (Finding true cluster $1$ and conclusion).} The second step provides a labelling $\hat g$ that coincides, up to permutation, with the aggregated labelling $g$ with high probability. The assumed separation from $1/2$ allows us to identify  cluster $1$: For $l=1,\ldots,k-1$, compute
\[ \hat N_{ll} := \frac{1}{{{\hat{g}^{-1}(l)|} \choose {2}}}\sum_{i<j,\, i,j\in \hat{g}^{-1}(l)} X_{ij}.\]
Since class sizes are of order $n/k$, an application of Bernstein's inequality
gives
\[ |\hat N_{ll} - N_{\sigma(l)\sigma(l)}|\le |\te|\1_{\sigma(l)=1} + O\Big(\sqrt{\frac{k}{n}\log{n}} \Big)\] for some permutation $\sigma$, with 
high probability. By (A3),  if $l\neq 1$, we have $|N_{ll}-1/2|\ge \kappa$. So w.h.p. there is exactly one diagonal element $\hat N_{ll}=\hat N_{\hat{\ell}\hat{\ell}}$ within $\kappa/2$ of $1/2$, since the conditions of the theorem imply
\[ |\te| + C\sqrt{\frac{k}{n}\log{n}} \le  \frac{\kappa}{2}. \]
The index $\hat\ell$ then identifies the first cluster of $N$---which is the aggregate cluster corresponding to clusters 1 and 2 defined by $M^\te$---with high probability. 
We can now apply the spectral algorithm for $k=2$ to the induced submatrix $(X_{ij})_{i,j\in\hat{g}^{-1}(\hat \ell)}$. Using the upper-bound part of Theorem \ref{thm-twocl-spa} with a number of nodes $|\hat{g}^{-1}(\hat\ell)| \asymp n/k$ leads to $E_\te[(\hat\te-\te)^2]\le Ck/(n\al_n)$, by 
observing that the event with high probability arising from the previous arguments (that is, the concentration result by  \cite{Lei:Rinaldo:2015:1} and Bernstein's inequalities)
holds with probability at least $1-1/n$.
\end{proof}

\end{document}